\documentclass[11pt,reqno]{amsart}
\usepackage{a4wide}
\usepackage{dsfont} 
\usepackage{amssymb}
\usepackage{comment}
\usepackage{graphicx}

\newtheorem{proposition}{Proposition}[section]
  \newtheorem{theorem}[proposition]{Theorem}
  \newtheorem{lemma}[proposition]{Lemma}
\theoremstyle{definition}
  \newtheorem{definition}[proposition]{Definition}

\newcommand{\cst}{\ifmmode\mathrm{C}^*\else{$\mathrm{C}^*$}\fi}
\newcommand{\st}{\;\vline\;} 
\newcommand{\tens}{\otimes} 
\newcommand{\atens}{\odot} 
\newcommand{\vtens}{\;\!\bar{\tens}\;\!} 
\newcommand{\dtens}{\;\!\dot{\tens}\;\!} 
\newcommand{\I}{\mathds{1}}
\newcommand{\flip}{\boldsymbol{\sigma}}
\newcommand{\comp}{\circ}
\newcommand{\ph}{\varphi}
\newcommand{\eps}{\varepsilon}
\newcommand{\id}{\mathrm{id}}
\newcommand{\qqquad}{\quad\qquad}
\newcommand{\hh}[1]{\widehat{#1}}
\newcommand{\scp}[2]{\left\langle#1,#2\right\rangle}

\newcommand{\GG}{\mathbb{G}}
\newcommand{\HH}{\mathbb{H}}
\newcommand{\CC}{\mathbb{C}}

\newcommand{\uu}{{\scriptscriptstyle\mathrm{u}}}
\newcommand{\nc}{{\text{---}\scriptscriptstyle{\|\cdot\|}}}
\newcommand{\op}{{\scriptscriptstyle\mathrm{op}}}

\newcommand{\sH}{\mathsf{H}}
\newcommand{\sM}{\mathsf{M}}
\newcommand{\sK}{\mathsf{K}}
\newcommand{\sN}{\mathsf{N}}

\newcommand{\cE}{\mathcal{E}}
\newcommand{\cP}{\mathcal{P}}
\newcommand{\cI}{\mathcal{I}}
\newcommand{\cF}{\mathcal{F}}
\newcommand{\Jnd}{\mathcal{J}}

\newcommand{\gN}{\mathfrak{N}}

\newcommand{\ww}{\mathrm{W}}
\newcommand{\vv}{\mathrm{V}}
\newcommand{\WW}{{\mathds{V}\!\!\text{\reflectbox{$\mathds{V}$}}}}
\newcommand{\Ww}{\mathds{W}}
\newcommand{\wW}{\text{\reflectbox{$\Ww$}}\:\!}

\DeclareMathOperator{\C}{C}
\DeclareMathOperator{\M}{M}
\DeclareMathOperator{\Mor}{Mor}
\DeclareMathOperator{\B}{B}
\DeclareMathOperator{\cL}{\mathcal{L}}
\DeclareMathOperator{\cK}{\mathcal{K}}
\DeclareMathOperator{\Linf}{\mathnormal{L}^\infty\!\!\;}
\DeclareMathOperator{\Ltwo}{\mathnormal{L}^2\!\!\;}
\DeclareMathOperator{\Lone}{\mathnormal{L}^1\!\!\;}
\DeclareMathOperator{\Rep}{Rep}
\DeclareMathOperator{\Lin}{\overline{span}}

\DeclareMathOperator*{\btens}{\,\tens\,} 

\newcommand{\pil}{\pi_l} 
\newcommand{\pir}{\pi_r} 
\newcommand{\tilpil}{\widetilde{\pi}_l} 
\newcommand{\corresp}[3]{\raisebox{-1ex}[0ex][1.3ex]{\scriptsize $#2$}\hspace{0.3ex}
  \framebox{$#1$}\hspace{0.3ex}
  \raisebox{-1ex}[0ex][1.3ex]{\scriptsize $#3$}}
\newcommand{\bicorresp}[4]{\raisebox{-1ex}[0ex][1.3ex]{\scriptsize $#2$}\hspace{0.3ex}
  \overset{#4}{\framebox{$#1$}}\hspace{0.3ex} \raisebox{-1ex}[0ex][1.3ex]{\scriptsize $#3$}}
\DeclareMathOperator{\Ind}{Ind}
\DeclareMathOperator{\IndR}{Ind_{\text{\tiny{\rm{R}}}}}
\DeclareMathOperator{\IndK}{Ind_{\text{\tiny{\rm{K}}}}}
\DeclareMathOperator{\IndV}{Ind_{\text{\tiny{\rm{V}}}}}

\numberwithin{equation}{section}

\begin{document}

\author{Mehrdad Kalantar}
\address{Department of Mathematics, University of Houston, Houston, TX 77204, USA}
\email{kalantar@math.uh.edu}

\author{Pawe{\l} Kasprzak}
\address{Department of Mathematical Methods in Physics, Faculty of Physics, University of Warsaw, Poland}
\email{pawel.kasprzak@fuw.edu.pl}

\author{Adam Skalski}
\address{Institute of Mathematics of the Polish Academy of Sciences, ul. \'Sniadeckich 8, 00-656 Warszawa, Poland}
\email{a.skalski@impan.pl}

\author{Piotr So{\l}tan}
\address{Department of Mathematical Methods in Physics, Faculty of Physics, University of Warsaw, Poland}
\email{piotr.soltan@fuw.edu.pl}

\begin{abstract}
In this paper we revisit the theory of induced representations in the setting of locally compact quantum groups. In the case of induction from open quantum subgroups, we show that constructions of Kustermans and Vaes are equivalent to the classical, and much simpler, construction of Rieffel. We also prove in general setting the continuity of induction in the sense of Vaes with respect to weak containment.
\end{abstract}

\title[Induction for locally compact quantum groups]{Induction for locally compact quantum groups revisited}




\maketitle

\section{Introduction}

The theory of induced representations, allowing one to manufacture in a canonical way a representation of a group out of that of a subgroup, dates back to the work of Frobenius for finite groups and was later developed by Mackey for locally compact groups. It has played a very important role in developing the general theory of group representations and has seen many applications in various areas of abstract harmonic analysis (we refer to the recent book \cite{KaniuthTaylor} for a precise description of the construction, the historical background and an exhaustive list of references). From the modern point of view a very important point in the history of the subject was the paper of Rieffel (\cite{rief1}), which phrased the construction in the language of Hilbert \cst-modules. This led to later generalizations of similar induction procedures to several operator algebraic contexts.

Thus it was natural that  in the beginning of the 21st century, with the widely accepted notion of locally compact \emph{quantum} groups emerging in the work of Kustermans and Vaes (\cite{KV}),  an interest has developed in the study of the theory of induced representations in the quantum context. In fact the two relevant papers were written respectively by Kustermans (\cite{KustermansInduced}) and Vaes (\cite{Vaes-induction}). In the first of these the construction was presented under the assumption of integrability of the natural action of the quantum subgroup on the larger quantum group (automatically satisfied for classical groups -- and later shown to be always true also in the quantum context in \cite{DeCommer}), in the second an alternative approach was proposed, explicitly disposing of the integrability constraint and also allowing for a more general framework of representations on Hilbert \cst-modules, as opposed to Hilbert spaces. Both constructions of Kustermans and Vaes are technically complicated; indeed even the classical procedure of Mackey in general requires overcoming certain measure-theoretic and topological difficulties.

The starting point of our paper is a classical observation that in the classical case if the subgroup from which we induce the representation is open, most of the technical problems disappear (see again \cite{KaniuthTaylor}). More importantly, one can apply then a very straightforward version of the Rieffel procedure, simply observing that if $H$ is an open subgroup of a locally compact group $G$, then one has a natural embedding of the group \cst-algebras $\cst(H)\subset\cst(G)$, which admits also a canonical conditional expectation $E::\cst(G)\to\cst(H)$. In a recent work \cite{KKS} three of the authors of the current paper have introduced the notion of an \emph{open quantum subgroup} of a locally compact quantum group $\GG$. One of the main results of \cite{KKS} and Section \ref{sect:RepTh} of this paper show that if $\HH$ is an open quantum subgroup of $\GG$, one still has an embedding $\cst(\HH)\subset\cst(\GG)$ and a canonical conditional expectation $E:\cst(\GG)\to\cst(\HH)$ (we write these relations as $\C_0^\uu(\hh{\HH})\subset\C_0^\uu(\hh{\GG})$ and $E:\C_0^\uu(\hh{\GG})\to\C_0^\uu(\hh{\HH})$). This opens the way to inducing a unitary representation of $\GG$ (in other words a representation of the \cst-algebra $\cst(\GG)$) from a unitary representation of $\HH$ (in other words a representation of the \cst-algebra $\cst(\HH)$) using the simple method of Rieffel and raises a natural question whether this leads to an object isomorphic to that obtained from the same starting data via the procedures of Kustermans and Vaes. This is the topic of this paper.

Using Kustermans's induction, we first prove the existence of the canonical embedding $\cst(\HH)\subset\cst(\GG)$ for any open quantum subgroup $\HH$ of a locally compact quantum group $\GG$, hence in particular dropping the coamenability assumption required in \cite[Section 5]{KKS}. This means one can in that case always induce a representation of $\GG$ (on a Hilbert \cst-module) from a representation of $\HH$ via Rieffel's method. We then prove that in fact in this case the resulting representation is canonically isomorphic to the ones obtained either via Kustermans's or Vaes's procedures. The key tool for that is an imprimitivity-type result, Theorem \ref{condition}, allowing us to identify Rieffel-induced representations of the larger quantum group via the existence of a covariant representation of the algebra of functions on the quantum homogeneous space $\GG/\HH$.

We note that in \cite[Section 11]{Vaes-induction} it was asserted that one could show that induction procedures of \cite{Vaes-induction} and \cite{KustermansInduced} are unitary equivalent in general, but that a proof would be highly non-trivial and would need to use the complete machinery of modular theory and properties of the unitary implementation of the action of the larger quantum group on the quantum homogeneous space. Our results provide a complete understanding of the situation in the case of open quantum subgroups. In particular, in this setting, the induction process (in any of the above senses) is rather simple, and does not require the more technical aspects of \cst-algebra or quantum group theories. It is worth recalling that, unsurprisingly, quantum subgroups of discrete quantum groups are always open and note that a simplified picture of the Vaes induction procedure in the discrete case was studied in \cite{VergniouxVoigt}.

The detailed plan of the paper is as follows: after introducing general notation in the last part of this section, in Section \ref{sect:Prel} we recall the basics of Hilbert \cst-modules and Rieffel's induction, introduce terminology and fundamental facts pertaining to locally compact quantum groups and discuss in reasonable detail the induction procedures developed in this context by Kustermans and Vaes. Section \ref{sect:RepTh} contains a generalization of Theorem 5.2 of \cite{KKS}, dropping the coamenability assumption for the inclusion $\C_0^\uu(\hh{\HH})\subset\C_0^\uu(\hh{\GG})$ (with $\HH$ being an open quantum subgroup of $\GG$) and thus opening the way to the Rieffel construction of induced representations in this context. A short Section \ref{Sec:imprim} establishes an imprimitivity type result, offering a method to recognize representations induced in the sense of Rieffel; this theorem is then applied in Section \ref{sect:equiv} to obtain the main results of the paper, namely the unitary equivalence of all the three induction procedures for the case of open quantum subgroups. A short appendix shows the continuity of the Vaes induction in the general case.

All scalar products are linear on the right, i.e.~in the second variable. The symbol $\tens$ denotes the completed Hilbert-space/Hilbert \cst-module/\cst-algebraic-minimal tensor product, and the algebraic tensor products will be denoted by $\atens$. The ultrawewak tensor product of von Neumann algebras will be denoted by $\vtens$. For a \cst-algebra $A$ we denote its multiplier algebra by $\M(A)$ and if $B$ is another \cst-algebra we write $\Mor(A,B)$ for the set of all nondegenerate $*$-homomorphisms from $A$ to $\M(B)$. Elements of $\Mor(A,B)$ are called \emph{morphisms} from $A$ to $B$ (cf.~\cite{Wo1}).
We will often use the leg numbering notation for operators acting on (multiple) tensor products of Hilbert spaces and elements of tensor products of algebras; $\flip$ will usually denote the tensor flip. For a set $X\subset\B(\sH)$, where $\sH$ is a Hilbert space, $X'$ denotes the commutant of $X$. Given a normal semifinite (n.s.f.) weight $\theta$ on a von Neumann algebra $\sN$ we write $\gN_\theta=\bigl\{x\in\sN\st\theta(x^*x)<+\infty\bigr\}$, $\Ltwo(\sN, \theta)$ for the associated GNS Hilbert space and $\Lambda_\theta:\gN_\theta\to\Ltwo(\sN,\theta)$ for the associated GNS map. Finally for a Hilbert space $\sH$ and two vectors $\xi,\eta\in\sH$ we write $\omega_{\xi,\eta}$ for the functional in $\B(\sH)_*$ given by $T\mapsto\scp{\xi}{T\eta}$ and $\omega_\xi=\omega_{\xi,\xi}$.

Finally a word on the quantum group notation is in order; although in \cite{KKS} we adopted the conventions of \cite{BaajSkandalis}, working with right Haar weights and right multiplicative unitaries as primary objects, here, mainly for compatibility with \cite{Vaes-induction} we stick to the left Haar weights, left multiplicative unitaries, etc.

\subsection*{Acknowledgement}

The third author was partially supported by the NCN (National Science Centre) grant 2014/14/E/ST1/00525. The second and fourth authors were partially supported by NCN (National Science Centre) grant no.~2015/17/B/ST1/00085.

\section{Preliminaries}\label{sect:Prel}

In this section we introduce basic facts and notations needed in the main body of the paper.

\subsection{Modules, correspondences and induction in the sense of Rieffel}

Let $B$ be a \cst-algebra. By a \emph{(Hilbert)\cst-module} over $B$ (usually called simply a \emph{Hilbert $B$-module}) we understand a right module $\cE$ over $B$ equipped with a $B$-valued scalar product satisfying natural requirements (see \cite{Lance}); by $\cL(\cE)$ we denote the \cst-algebra of \emph{adjointable} operators on $\cE$. Every pair of vectors $\xi,\eta\in\cE$ defines an operator $\Theta_{\xi,\eta}\in\cL(\cE)$ by the formula
\[
\Theta_{\xi,\eta}(\zeta)=\xi\scp{\eta}{\zeta},\qqquad\zeta\in\cE.
\]
The closure of the span of operators of the form $\Theta_{\xi,\eta}$ forms a \cst-subalgebra of $\cL(\cE)$ called the algebra of \emph{compact} operators on $\cE$ and denoted by $\cK(\cE)$. We then have $\cL(\cE)=\M(\cK(\cE))$. If in addition we have a non-degenerate $*$-homomorphism from another $\cst$-algebra $A$ to $\cL(\cE)$ (in other words, a \emph{representation of $A$ on $\cE$}) we call $\cE$ an \emph{$A$-$B$-correspondence}. We will sometimes use without any comment the internal tensor product of \cst-correspondences and the external product of Hilbert modules. In particular if $A$ and $B$ are \cst-algebras and $\cE$ is a Hilbert $B$-module then $A\tens\cE$ denotes the natural completed Hilbert $(A\tens{B})$-module (cf.~\cite[Chapter 4]{Lance}).

We will also occasionally need the notion of von Neumann modules, replacing \cst-algebras with von Neumann algebras and adding normality conditions for the respective actions (see the Appendix of \cite{Vaes-induction} for the details). A key notion, also introduced in \cite{Vaes-induction}, is that of a strict $*$-homomorphism: let $\sM$ be a von Neumann algebra and $\cE$ a Hilbert module over a \cst-algebra $B$. A map $\pi:\sM\to\cL(\cE)$ is \emph{strict} if it is continuous with respect to the strong${}^*$ topology on the unit ball of $\sM$ and strict topology on $\cL(\cE)$ (cf.~\cite[Page 11]{Lance}).

Suppose $A$ is a \cst-algebra, $B$ is a \cst-subalgebra of $A$ and and $E:A\to{B}$ is a \emph{conditional expectation} (i.e.~a norm-one projection -- it is automatically completely positive). We will now recall Rieffel's process of inducing a representation of $A$ from a representation of $B$ (on a Hilbert module) presented in \cite{rief1}.

Let $\pi:B \to\cL(\cE)$ be a representation of $B$ on a Hilbert module $\cE$ over some auxiliary \cst-algebra $C$. On the $B$-balanced algebraic tensor product $A\atens_B\cE$ define an $A$-valued inner product:
\[
\scp{a\tens{v}}{b\tens{w}}=\scp{v}{E(a^*b)w},\qqquad{a,b}\in{A},\:v,w\in\cE.
\]
Then the action of $A$ on $A\atens_B\cE$ defined by $a(b\tens{v})=ab\tens{v}$ gives a representation $\IndR(\pi)$ of $A$ on the \cst-$C$-module $\IndR(\cE)$ obtained from $A\atens_B\cE$ via the usual separation/completion procedure. The representation $\IndR(\pi)$ is called the representation \emph{induced from $\pi$ in the sense of Rieffel}. The above construction can be easily phrased in the language of internal tensor product of Hilbert \cst-modules.

\subsection{Locally compact quantum groups}

As explained in the introduction we will use here the \emph{left} conventions of \cite{Vaes-induction}. Thus $\GG$ denotes a \emph{locally compact quantum group} in the sense of \cite{KV}, a virtual object studied via the associated operator algebras: the von Neumann algebra $\Linf(\GG)$ (``essentially bounded measurable functions on $\GG$''), the \cst-algebra $\C_0(\GG)\subset\Linf(\GG)$ (``continuous functions on $\GG$ vanishing at infinity''), and its universal version $\C_0^\uu(\GG)$. Each of these is equipped with a \emph{coproduct}: we have a unital normal coassociative $*$-homomorphism $\Delta_\GG: \Linf(\GG)\to\Linf(\GG)\vtens\Linf(\GG)$ which restricts to $\Delta_\GG\in\Mor(\C_0(\GG),\C_0(\GG)\tens\C_0(\GG))$ and also a corresponding $*$-homomorphism $\Delta_\GG^\uu\in\Mor(\C_0^\uu(\GG),\C_0^\uu(\GG)\tens\C_0^\uu(\GG))$. The canonical surjective morphism from $\C_0^\uu(\GG)$ onto $\C_0(\GG)$ will be denoted by $\Lambda_\GG$. If $\Lambda_\GG$ is injective, we say that $\GG$ is \emph{coamenable}. The \emph{left} (respectively, \emph{right}) Haar weight on $\Linf(\GG)$ will be denoted by $\ph_\GG$ (respectively, $\psi_\GG$) and $\Ltwo(\GG)$ will denote the GNS Hilbert space of the left Haar weight. We will always assume that $\Linf(\GG)$ and $\C_0(\GG)$ are represented on $\Ltwo(\GG)$. The key object carrying all the information about $\GG$ is the \emph{left regular representation} $\ww^\GG\in\B(\Ltwo(\GG)\tens\Ltwo(\GG))$ implementing the comultiplication:
\[
\Delta_\GG(x)=(\ww^\GG)^*(\I\tens{x})\ww^\GG,\qqquad{x}\in\Linf(\GG).
\]
This operator is also called the \emph{Kac-Takesaki operator} or, less formally, the \emph{multiplicative unitary} of $\GG$. It carries also the information about the \emph{dual locally compact quantum group} $\hh{\GG}$ (see \cite[Section 8]{KV}): on one hand we have $\C_0(\hh{\GG})= \bigl\{(\omega\tens\id)(\ww^\GG)\st\omega\in\B(\Ltwo(\GG))_*\bigr\}^{\nc}$, and on the other $\ww^\GG\in\Linf(\GG)\vtens\Linf(\hh{\GG})$ (or more precisely $\ww^\GG\in\M(\C_0(\GG)\tens\C_0(\hh{\GG}))$). Note that in particular $\Linf(\hh{\GG})=\C_0(\hh{\GG})''$ is canonically represented on $\Ltwo(\GG)$. Moreover the multiplicative unitary associated with $\hh{\GG}$ is given by formula $\ww^{\hh{\GG}}=\flip(\ww^\GG)^*$.

The element $\ww^\GG\in\M(\C_0(\GG)\tens\C_0(\hh{\GG}))$ admits a universal version, $\WW^\GG\in\M(\C_0^\uu(\GG)\tens\C_0^\uu(\hh{\GG}))$, such that $\ww^\GG= (\Lambda_{\GG}\tens\Lambda_{\hh{\GG}})(\WW^\GG)$. We may also consider natural ``one-sided'' reduced/universal versions, $\wW^\GG=(\Lambda_{\GG}\tens\id)(\WW^\GG)$ and $\Ww^\GG=(\id\tens\Lambda_{\hh{\GG}})(\WW^\GG)$. The predual of $\Linf(\GG)$ is denoted $\Lone(\GG)$; it is a Banach algebra in a canonical way and we have a canonical left regular representation $\lambda^\uu:\Lone(\GG)\to\C_0^\uu(\hh{\GG})$ given by the formula $\lambda^\uu(\omega)=(\id\tens\omega)(\Ww^\GG)$. The Banach algebra $\Lone(\GG)$ admits a dense subalgebra $\Lone_{\#}(\GG)$ which carries a natural involution, and the map $\lambda^\uu$ restricted to the latter becomes a $*$-homomorphism of Banach $*$-algebras. Occasionally we will also need the \emph{right multiplicative unitary} $\vv^\GG\in\Linf(\hh{\GG})'\vtens\Linf(\GG)$, defined as $\vv^\GG=
(\hh{J}\otimes \hh{J})W^{\hat\GG}(\hh{J}\otimes \hh{J})$, where $\hh{J} $ 
denotes 
the modular conjugations associated with the pair
$(\Linf(\GG),\ph_\GG)$, and its dual version $\vv^{\hh{\GG}}\in\Linf(\GG)'\vtens\Linf(\hh{\GG})$.

A \emph{unitary representation of $\GG$} on a Hilbert module $\cE$ (usually shortened to simply ``representation'') is a unitary $U\in\M(\C_0(\GG)\tens\cK(\cE))=\cL(\C_0(\GG)\tens\cE)$ such that $(\Delta_\GG\tens\id)U=U_{13}U_{23}$. We then write $U\in\Rep(\GG,\cE)$. Unitary representations of $\GG$ are in one-to-one correspondence with representations of the \cst-algebra $\C_0^\uu(\hh{\GG})$: if $\phi$ is a representation of $\C_0^\uu(\hh{\GG})$ on $\cE$ then
\[
U=(\id\tens\phi)(\wW^\GG)
\]
is a representation of $\GG$ on $\cE$ and every $U\in \Rep(\GG,\cE)$ comes from this construction for a unique representation $\phi_U$ of $\C_0^\uu(\hh{\GG})$.

Given two representations $U$ and $V$ on Hilbert modules $\cE_U$ and $\cE_V$ (over the same \cst-algebra) we say that $U$ is \emph{contained} in $V$ if there exists a projection $P\in\cL(\cE_V)$ such that $U\cong(\id\tens{P})V(\id\tens{P})$, where $\cong$ denotes the self-explanatory relation of unitary equivalence. We say that $U$ is \emph{weakly contained} in $V$ if $\ker{\phi_V}\subset\ker{\phi_U}$; we denote it by writing $U\preccurlyeq V$.

If $\sM$ is a von Neumann algebra then by a (right) action of $\GG$ on $\sM$ we understand an injective unital normal $*$-homomorphism $\alpha:\sM\to\sM\vtens\Linf(\GG)$ satisfying the \emph{action equation} $(\alpha\tens\id)\comp\alpha=(\id\tens\Delta_\GG)\comp\alpha$. If we write $\sM$ in the form $\Linf(\mathbb{X})$ for some classical or ``quantum'' space $\mathbb{X}$ we speak simply about actions of $\GG$ on $\mathbb{X}$. The \emph{crossed product} of $\sM$ by the action $\alpha$, denoted $\sM\rtimes_\alpha\GG$, is the von Neumann algebra generated inside $\sM\vtens\B(\Ltwo(\GG)$ by $\alpha(\sM)$ and $\I\tens\Linf(\hh{\GG})$. Actions and crossed products admit also natural left versions. It was shown in \cite{Vaes-implementation} that every left action $\alpha$ of a locally compact quantum group $\GG$ on a von Neumann algebra $\sM$ admits a \emph{canonical unitary implementation} $\Upsilon\in\Linf(\GG)\vtens\B(\sH)$, where $\sH$ is a space on which $\sM$ acts in the standard form. This means that
\[
\alpha(m)=\Upsilon^*(\I\tens{m})\Upsilon,\qqquad{m}\in\sM.
\]
Moreover $\Upsilon$ is a representation of $\GG$ on $\sH$

We will also need the notion of \emph{\cst-algebraic actions of locally compact quantum groups} and of \emph{(compatible) actions of locally compact quantum groups on Hilbert modules}, both in the \cst-algebraic and von Neumann algebraic settings. Here again we refer to the Appendix of \cite{Vaes-induction} for the details, recalling only that by a compatible (right) action of $\GG$ on a Hilbert $B$-module $\cE$ for a \cst-algebra $B$ we understand a pair of maps $\alpha_B:B\to\M(B\tens\C_0(\GG))$ and $\alpha_\cE:\cE\to\M(\cE\tens\C_0(\GG))$ satisfying natural compatibility conditions (\cite[Definition 12.2]{Vaes-induction}.

\subsection{Closed and open quantum subgroups}\label{Closed/open}

Given two locally compact quantum groups $\GG$ and $\HH$, a morphism $\Pi$ from $\HH$ to $\GG$ (written $\Pi:\HH\to\GG$) is represented via either of the following three objects
\begin{itemize}
\item a \emph{Hopf $*$-homomorphism}, i.e.~an element $\pi^\uu\in \Mor(\C_0^\uu(\GG),\C_0^\uu(\HH))$ intertwining the respective coproducts:
\[
(\pi^\uu\tens\pi^\uu)\comp\Delta_{\GG}^\uu=\Delta_{\HH}^\uu\comp\pi^\uu;
\]
\item a \emph{bicharacter} from $\HH$ to $\GG$, i.e.~a unitary $V\in\M(\C_0(\HH))\tens\C_0(\hh{\GG}))$ such that
\[
\begin{split}
(\id\tens\Delta_{\hh{\GG}})V&=V_{13}V_{12},\\
(\Delta_{\HH}\tens\id)V&=V_{13}V_{23};
\end{split}
\]
\item a \emph{right quantum group homomorphism}, i.e.~an action $\alpha:\Linf(\GG)\to\Linf(\GG)\vtens\Linf(\HH)$ of $\HH$ on $\GG$ such that
\[
(\Delta_\GG\tens\id)\comp\alpha=(\id\tens\alpha)\comp\Delta_\GG.
\]
\end{itemize}
The relationships between Hopf $*$-homomorphisms, bicharacters and right quantum group homomorphisms are described in \cite{SLW12,DKSS} (note there is a right/left change in the notation in the treatment of bicharacters above). One of these is that $V=(\Lambda_{\HH}\comp\pi\tens\id)(\Ww^{\GG})$, another is
\[
\alpha(x)=V^*(\I\tens{x})V,\qqquad{x}\in\Linf(\GG).
\]

Each homomorphism $\Pi:\HH\to\GG$ admits a unique \emph{dual} $\hh{\Pi}:\hh{\GG}\to\hh{\HH}$. If $\pi^\uu$ is the Hopf $*$-homomorphism describing $\Pi$ then the Hopf $*$-homomorphism corresponding to $\hh{\Pi}$ is denoted by $\hh{\pi}^\uu$. It is uniquely determined by the relation
\begin{equation}\label{WW}
(\id\tens\hh{\pi}^\uu)(\WW^{\HH})=(\pi^\uu\tens\id)(\WW^{\GG}).
\end{equation}
On the level of bicharacters $\hh{\Pi}$ is described by $\hh{V}=\flip(V)^*$.

A Hopf $*$-homomorphism $\pi^\uu\in\Mor(\C_0^\uu(\GG),\C_0^\uu(\HH))$ may admit a \emph{reduced version}, i.e.~an element $\pi\in\Mor(\C_0(\GG),\C_0(\HH))$ such that $\Lambda_\HH\comp\pi^\uu=\pi\comp\Lambda_\GG$. It may then happen that $\pi$ admits an extension to a normal $*$-homomorphism $\Linf(\GG)\to\Linf(\HH)$. Our notation will not distinguish between $\pi$ and its extension to $\Linf(\GG)$.

\begin{definition}\label{Def:Sbgrp}
A homomorphism from $\HH$ to $\GG$ described by a Hopf $*$-homomorphism $\pi^\uu\in\Mor(\C_0^\uu(\GG),\C_0^\uu(\HH))$ identifies $\HH$ with a closed quantum subgroup of $\GG$ (in the sense of Vaes) if there exists a reduced version $\hh{\pi}$ of $\hh{\pi}^\uu$ which extends to an injective normal map $\Linf(\hh{\HH})\to\Linf(\hh{\GG})$. We often simply say that $\HH$ is a closed quantum subgroup of $\GG$.
\end{definition}

In the situation of Definition \ref{Def:Sbgrp} the normal injection $\Linf(\hh{\HH})\to\Linf(\hh{\GG})$ automatically intertwines comultiplications and in fact existence of such an injection is equivalent to $\HH$ being a closed quantum subgroup of $\GG$. Moreover on the level of $\C_0^\uu(\GG)$ the map $\pi^\uu$ is then a surjective $*$-homomorphism $\C_0^\uu(\GG)\to\C_0^\uu(\HH)$.

Let $\HH$ be a closed quantum subgroup of $\GG$. The \emph{(measured) quantum homogeneous space} $\GG/\HH$ is defined by setting
\[
\Linf(\GG/\HH)=\bigl\{x\in\Linf(\GG)\st\alpha(x)=x\tens\I\bigr\},
\]
where $\alpha$ is the right quantum group homomorphism associated to the inclusion $\HH\hookrightarrow\GG$. The coproduct $\Delta_\GG$ restricts to a left action of $\GG$ on $\GG/\HH$. We denote this restriction by $\rho_{\GG/\HH}$.

\begin{definition}\label{Definition:open}
A homomorphism from $\HH$ to $\GG$ corresponding to a Hopf $*$-homomorphism $\pi^\uu\in\Mor(\C_0^\uu(\GG),\C_0^\uu(\HH))$ identifies $\HH$ with an open quantum subgroup of $\GG$ if there exists a reduced version $\pi$ of $\pi^\uu$ which extends to a surjective normal map $\pi:\Linf(\GG)\to\Linf(\HH)$. We often simply say that $\HH$ is an open quantum subgroup of $\GG$.
\end{definition}

An open quantum subgroup of a locally compact quantum group is automatically closed (\cite[Section 3]{KKS}).

The key object when dealing with open quantum groups is the \emph{central support} (in the terminology of \cite{KKS}, in literature it is often called the \emph{central cover}) of the normal surjective morphism $\pi:\Linf(\GG)\to\Linf(\HH)$. It is the smallest central projection mapped to $\I$ by $\pi$ and it is denoted by $\I_\HH$. It was shown in \cite{KKS} that $\I_\HH$ is a \emph{group-like projection}, i.e.~$\Delta_\GG(\I_\HH)(\I_\HH\tens\I)=\I_\HH\tens\I_\HH$. Moreover we can then identify $\Ltwo (\HH)$ with $\I_\HH\Ltwo(\GG)$ and $\Linf(\HH)$ with $\I_\HH\Linf(\GG)$; we will do so without further comment. We call $\I_\HH$ simply the \emph{support projection of $\HH$} and note that it belongs to $\Linf(\GG/\HH)$.

\subsection{Induction in the sense of Kustermans}\label{kusrep}

In this subsection we briefly recall Kustermans's notion of induction for unitary representations of closed quantum subgroups of locally compact quantum groups on Hilbert spaces. We focus on establishing the notation and terminology; for the details of the construction we refer the reader to \cite{KustermansInduced}.

Let $\GG$ be a locally compact quantum group and $\HH$ a closed quantum subgroup of $\GG$. Let $U\in \Linf(\HH)\vtens \B(\sK)$ be a unitary representation of $\HH$ on a Hilbert space $\sK$. Fix an n.s.f.~weight $\theta$ on $\Linf(\GG/\HH)$ and denote by $\sH_\theta$ the corresponding GNS Hilbert space. The GNS representation will be denoted by $\pi_\theta:\Linf(\GG/\HH)\to\B(\sH_\theta)$. We often identify $\Linf(\GG/\HH)$ with its image under the GNS map and omit $\pi_\theta$ when there is no danger of confusion.

Let $\sH$ be a fixed Hilbert space. Recall from Subsection \ref{Closed/open} the canonical right action of $\HH$ on $\GG$ denoted by $\alpha$. Following the notation introduced in \cite[Definition 4.1]{KustermansInduced}, we let
\[
\cP_\sH=\bigl\{X\in\B(\sH)\vtens\Linf(\GG)\vtens\B(\sK)\st(\id\tens\alpha\tens\id)X=U_{34}^*X_{124}\bigr\}.
\]
When $\sH=\CC$ then we shall simply write $\cP$ for $\cP_\CC$. Observe that if $X\in\cP$ and $y\in\Linf(\GG/\HH)$ then $(y\tens\I)X\in\cP$.

On the algebraic tensor product $\cP\atens(\sH_\theta\tens\sK)$ define the (pre-)inner product
\[
\scp{X\tens{w}}{Y\tens{v}}=\scp{w}{(\id\tens\pi_\theta\tens\id)(X^*Y)v},\qqquad{X,Y}\in\cP,\:{w,v}\in\sH_\theta\tens\sK.
\]
and let $\IndK(\sK)$ be the Hilbert space obtained via the separation/completion procedure from this (pre-)inner product.

As the notation suggests, $\IndK(\sK)$ will be the Hilbert space on which the induced representation of $U$ will act. Proposition 4.6 of \cite{KustermansInduced} shows that to every element $X\in\cP_\sH$ one can associate in a canonical way an operator $X_*\in\B(\sH\tens\sH_\theta\tens\sK, \sH\tens\IndK(\sK))$. We will later use the following properties stated in \cite[Results 4.9, Notation 4.7]{KustermansInduced}:
\begin{subequations}\label{props0}
\begin{align}
(X_*)^*Y_*&=(\id\tens\pi_\theta\tens\id)(X^*Y),&&X,Y\in\cP_\sH,\label{props1}\\
(XZ)_*&= X_*(\id\tens\pi_\theta\tens\id)(Z),&&X\in\cP_\sH,\:Z\in\B(\sH)\vtens\Linf(\GG/\HH)\vtens\B(\sK).\label{props2}
\end{align}
\end{subequations}
In the special case when $\sH=\CC$ the definition of $X_*$ simplifies: for $X\in\cP$ and $w\in\sH_\theta\tens\sK$ the element $X_*w$ is the class of $X\tens{w}$ in $\IndK(\sK)$ (denoted $X\dtens{w}$ in \cite{KustermansInduced}).

The representation $\IndK(U)$ \emph{induced from $U$ in the sense of Kustermans} is defined by the linear extension of the formula
\[
\IndK(U)^*(\eta\tens{X_*w})=\bigl((\Delta_\GG\tens\id)(X)\bigr)_*\Upsilon_{12}^*(\eta\tens{w})
\]
for $X\in\cP$, $\eta\in\Ltwo(\GG)$, and $w\in\sH_\theta\tens\sK$, where $\Upsilon\in\Linf(\GG)\vtens\B(\sH_\theta)$ is the canonical unitary implementation of the action $\rho_{\GG/\HH}$ of $\GG$ on $\Linf(\GG/\HH)$ (note that $(\Delta_\GG\tens\id)(X)\in\cP_{\Ltwo(\GG)}$).

Recall that, as stated in the introduction, results of \cite{DeCommer} guarantee together with these of \cite[Section 7]{KustermansInduced} that $\IndK(U)$ is indeed a unitary representation of $\GG$.

\subsection{Induction in the sense of Vaes}\label{vaesrep}

In this subsection we recall the notion of induction due to Vaes, starting from a unitary representation of a locally compact quantum group $\HH$ on a Hilbert module $\cE$. As before we assume that $\HH$ is a closed quantum subgroup of $\GG$. Here the basic idea comes from identifying representations of $\GG$ with certain $\Linf(\hh{\GG})-\Linf(\hh{\GG})$ correspondences equipped in addition with a \emph{bicovariant} action of the algebra $\Linf(\GG)'$ (see \cite[Proposition 3.7]{Vaes-induction}).

We begin however with defining an auxiliary object, the so-called \emph{imprimitivity bimodule}. Let $\HH$ be a closed quantum subgroup of $\GG$ and let $\hh{\pi}:\Linf(\hh{\HH})\to\Linf(\hh{\GG})$ be the corresponding inclusion. We will need its ``commutant'' version $\hh{\pi}':\Linf(\hh{\HH})'\to\Linf(\hh{\GG})'$:
\[
\hh{\pi}'(x)=\hh{J}_\GG\hh{\pi}\bigl(\hh{J}_\HH{x}\hh{J}_\HH\bigr)\hh{J}_\GG,\qqquad{x}\in\Linf(\HH)',
\]
where $\hh{J}_\GG$ and $\hh{J}_\HH$ denote the modular conjugations associated to pairs $\bigl(\Linf(\hh{\GG}),\ph_{\hh{\GG}}\bigr)$ and $\bigl(\Linf(\hh{\HH}),\ph_{\hh{\HH}}\bigr)$ respectively.

The \emph{imprimitivity bimodule} is the space
\[
\cI={\bigl\{v\in\B(\Ltwo(\HH),\Ltwo(\GG))\st}{vx}=\hh{\pi}'(x)v\text{ for all }x\in\Linf(\hh{\HH})'\bigr\}.
\]

With the natural scalar product $\scp{v}{w}=v^*w$ and left and right actions $\cI$ becomes a von Neumann $\GG\ltimes\Linf(\GG/\HH)-\Linf(\hh{\HH})$ correspondence.

Furthermore the map $\alpha_\cI:\cI\to\cI\tens\Linf(\hh{\GG})$ given by
\[
\alpha_\cI(v)=\vv^{\hh{\GG}}(v\tens\I)\bigl((\id\tens\hh{\pi})\bigl(\vv^{\hh{\HH}}\bigr)\bigr)^*,\qqquad{v}\in\cI,
\]
together with $\Delta_{\hh{\HH}}$ defines a compatible action of $\hh{\GG}$ on $\cI$.

Moreover, when we equip the crossed product $\GG\ltimes\Linf(\GG/\HH)$ with the dual action of $\hh{\GG}$, the left module action of $\GG\ltimes\Linf(\GG/\HH)$ on $\cI$ becomes covariant in a natural sense.

Now let $X\in\cL(\C_0(\HH)\tens\cE)$ be a unitary representation of $\HH$ on a Hilbert module $\cE$ over a \cst-algebra $B$. Lemma 4.5 of \cite{Vaes-induction} yields the existence of a unique strict $*$-homomorphism $\pil:\Linf(\hh{\HH})\to\cL(\Ltwo(\GG)\tens\cE)$ satisfying
\begin{equation}\label{Defpli}
(\id\tens\pil)\ww^\HH=(\id\tens\hh{\pi})(\ww^\HH)_{12}X_{13}.
\end{equation}

If we now define $\pir:\Linf(\hh{\GG})\to\cL(\Ltwo(\GG)\tens\cE)$ by the formula
\[
\pir(x)=\hh{J}_{\GG}x^*\hh{J}_{\GG}\tens\I,\qqquad{x}\in\Linf(\hh{\GG})
\]
and a map $\gamma:\Linf(\GG)'\to\cL(\Ltwo(\GG)\tens\cE)$ by
\[
\gamma(y)=y\tens\I,\qqquad{y}\in\Linf(\GG)'
\]
we obtain a so-called a \emph{bicovariant $B$-correspondence}
\begin{equation}\label{eq.ourcorresp}
\bicorresp{\Ltwo(\GG)\tens\cE}{\Linf(\hh{\HH})}{\Linf(\hh{\GG})}{\Linf(\GG)'}\:.
\end{equation}
Then the results of the Appendix of \cite{Vaes-induction} imply that one can define the Hilbert $B$-module ${\widetilde{\cF}}=\cI\btens\limits_{\pil}(\Ltwo(\GG)\tens\cE)$ and left and right module actions on ${\widetilde{\cF}}$ such that we get a $B$-correspondence
$\corresp{\widetilde{\cF}}{\GG\ltimes\Linf(\GG/\HH)}{\Linf(\hh{\GG})}$ which is equipped with the \emph{product action} $\alpha_{\widetilde{\cF}}$ of $\hh{\GG}$, constructed out of $\alpha_\cI$ and $\alpha_{\Ltwo(\GG)\tens\cE}$, where $\alpha_{\Ltwo(\GG)\tens\cE}$ is the action of $\hh{\GG}$ on $\Ltwo(\GG)\tens\cE$ defined by
\[
\alpha_{\Ltwo(\GG)\tens\cE}(\xi)=\vv^{\hh{\GG}}_{13}(\xi\tens\I),\qqquad\xi\in{\Ltwo(\GG)\tens\cE}.
\]
(cf.~\cite[Paragraph following Lemma 4.5]{Vaes-induction}). The action $\alpha_{\widetilde{\cF}}$ is given by the formula
\[
\alpha_{\widetilde{\cF}}(v\btens_{\pil}\zeta)=\alpha_\cI(v)\btens_{\pil\tens\id}\alpha_{\Ltwo(\GG)\tens\cE}(\zeta),
\qqquad{v}\in\cI,\:\zeta\in\Ltwo(\GG)\tens\cE.
\]
(cf.~\cite[Proposition 12.13]{Vaes-induction}).

The action $\alpha_{\widetilde{\cF}}$ yields a representation $\pi':\Linf(\GG)'\to\cL(\widetilde{\cF})$ and $\hh{U}\in\Rep(\hh{\GG},\widetilde{\cF})$ such that
\begin{equation}\label{hatU}
\hh{U}=(\pi'\tens\id)(\vv^{\hh{\GG}})
\end{equation}
and
\begin{equation}\label{hatU-alphaF}
\hh{U}(\Omega\tens a)=\alpha_{\widetilde{\cF}}(\Omega)(1 \otimes a)
\end{equation}
for all $\Omega\in\widetilde{\cF}$ and $a\in\C_0(\hh{\GG})$ (see \cite[Definition A.2 and the following paragraph]{Vaes-induction}).

Covariance of $Y$ with respect to $\pil$ and $\pir$ in the sense of \cite[Remark 3.6]{Vaes-induction} yields the \emph{bicovariant} $B$-correspondence
\[
\bicorresp{{\widetilde{\cF}}}{\Linf(\hh{\GG})}{\Linf(\hh{\GG})}{\Linf(\GG)'}\:.
\]
Now \cite[Proposition 3.7]{Vaes-induction} shows the existence of a canonically determined Hilbert $B$-module $\Ind(\cE)$ together with a unitary representation $\Ind(X)\in\Rep(\GG,\Ind(\cE))$ such that
\begin{equation}\label{box-iden}
\bicorresp{{\widetilde{\cF}}}{\Linf(\hh{\GG})}{\Linf(\hh{\GG})}{\Linf(\GG)'}\;\cong\;
\bicorresp{\Ltwo(\GG)\tens\Ind(\cE)}{\Linf(\hh{\GG})}{\Linf(\hh{\GG})}{\Linf(\GG)'}
\end{equation}
as bicovariant correspondences. In what follows we shall use the symbol $\IndV(\cdot)$ to denote the objects defined above. Note that the fact that $\widetilde{\cF}$ is a left $\GG\ltimes\Linf(\GG/\HH)$-module implies in particular existence of a map
\begin{equation}\label{covrepVaes}
\rho:\Linf(\GG/\HH)\longrightarrow\cL(\IndV(\cE))
\end{equation}
satisfying the covariance relation: $\Ind(X)^*\bigl(\I\tens\rho(x)\bigr)\Ind(X)=(\id\tens\rho)\Delta_\GG(x)$ for all $x\in\Linf(\GG/\HH)$.

The Hilbert $B$-module $\IndV(\cE)$ is called the Hilbert $B$-module \emph{induced in the sense of Vaes} from $\cE$ and the unitary representation $\IndV(X)$ is called the representation \emph{induced} by $X$ (also in the sense of Vaes).

\section{Representation-theoretic characterization of open quantum subgroups}\label{sect:RepTh}

In this section we prove that for an open quantum subgroup $\HH$ of a locally compact quantum group $\GG$ we have a canonical \cst-inclusion $\C_0^\uu(\hh{\HH}) \subset\C_0^\uu(\hh{\GG})$. We remark that this was previously proved by the  first three authors under the coamenability assumption on $\GG$ (see \cite[Theorem 5.2]{KKS}). The proof given here is completely different from the one in \cite{KKS}. In fact, in the latter we showed through direct computations that every positive definite function in $\Linf(\HH)$ is also positive definite when regarded as an element in $\Linf(\GG)$. As in the general, non-coamenable, case the criterion for positive-definiteness used in \cite{KKS} does not hold (see \cite{DawsSalmi}), that proof cannot be extended to arbitrary locally compact quantum groups. In the present approach we utilize Kustermans's theory of induced representations.

In view of the 1-1 correspondence between unitary representations of a locally compact quantum group $\GG$ and representations of $\C_0^\uu(\hh{\GG})$, the result mentioned above, together with the existence of a suitable conditional expectation, allows us to apply classical construction of inducing representations from \cst-subalgebras due to Rieffel \cite{rief1} to the case of unitary representations of open quantum subgroups.

In what follows $\GG$ is a locally compact quantum group and $\HH$ is an open quantum subgroup of $\GG$. Recall that we denote by $\alpha:\Linf(\GG)\to\Linf(\GG)\vtens\Linf(\HH)$ the canonical (right) action of $\HH$ on $\GG$ associated to the inclusion $\HH\hookrightarrow\GG$. Recall also that we write $\rho_{\GG/\HH}$ for the restriction of $\Delta_\GG$ to $\Linf(\GG/\HH)$, giving the canonical left action of $\GG$ on $\GG/\HH$.

Let $\theta$ be the disintegration of $\ph_\GG$, that is the unique n.s.f.~weight on $\Linf(\GG/\HH)$ satisfying
\begin{equation}\label{disintegr}
\theta\bigl((\id\tens\ph_\HH)(\alpha(x))\bigr)=\ph_\GG(x),\qqquad{x}\in\Linf(\GG)^+
\end{equation}
(see \cite[Proposition 8.7]{KustermansInduced}).

\begin{lemma}\label{theta-inv}
The weight $\theta$ is invariant, that is
\[
(\id\tens\theta)(\rho_{\GG/\HH}(x))=\theta(x)\I
\]
for all $x\in\Linf(\GG/\HH)^+$. Moreover $\theta(\I_\HH)$ is finite and non-zero.
\end{lemma}

\begin{proof}
Let $\omega\in\Lone(\GG)$ be a state. Define a (normal semifinite) weight $\theta_\omega$ on $\Linf(\GG/\HH)$ by
\[
\theta_\omega(y)=\theta\bigl((\omega\tens\id)(\rho_{\GG/\HH}(y))\bigr),\qqquad{y}\in\Linf(\GG/\HH)^+.
\]
For $x\in\Linf(\GG)^+$ we have
{\allowdisplaybreaks
\begin{align*}
\theta_\omega\bigl((\id\tens\ph_\HH)\alpha(x)\bigr)&=\theta\Bigl((\omega\tens\id)(\rho_{\GG/\HH}\bigl((\id\tens\ph_\HH)(\alpha(x))\bigr))\Bigr)\\
&=\theta\Bigl((\omega\tens\id)(\Delta_\GG\bigl((\id\tens\ph_\HH)\alpha(x)\bigr))\Bigr)\\
&=\theta\bigl((\omega\tens\id)((\id\tens\id\tens\ph_\HH)(\Delta_\GG\tens\id)(\alpha(x)))\bigr)\\
&=\theta\bigl((\omega\tens\id)((\id\tens\id\tens\ph_\HH)(\id\tens\alpha)(\Delta_\GG(x)))\bigr)\\
&=\theta\bigl((\id\tens\ph_\HH)((\omega\tens\id\tens\id)(\id\tens\alpha)(\Delta_\GG(x)))\bigr)\\
&=\theta\Bigl((\id\tens\ph_\HH)(\alpha\bigl((\omega\tens\id)(\Delta_\GG(x)))\bigl)\Bigr)\\
&=\ph_\GG\bigl((\omega\tens\id)(\Delta_\GG(x))\bigr) = \ph_\GG(x).
\end{align*}
}By uniqueness of $\theta$ we get $\theta_\omega=\theta$, which means that $\theta$ is invariant under the action $\rho_{\GG/\HH}$.

For the second assertion, note first that since the support projection $\I_\HH$ of $\HH$ is minimal in $\Linf(\GG/\HH)$ by \cite[Proposition 3.2]{KKS}, it follows that $\theta(\I_\HH)<+\infty$ (otherwise $\theta$ would not be semifinite). Suppose now that $\theta(\I_\HH) = 0$. Then the invariance of $\theta$ yields $\theta\bigl((\omega\tens\id)(\Delta_\GG(\I_\HH))\bigr)=0$ for all $\omega\in\Lone(\GG)$. Since
$\bigl\{(\omega\tens\id)(\Delta_\GG(\I_\HH))\st\omega\in\Lone(\GG)\bigr\}$ is a weak${}^*$-dense ideal in $\Linf(\GG/\HH)$ by \cite[Paragraph following Eq.~(3.6) in the proof of Theorem 3.3]{KKS} the last condition implies that $\theta=0$, which contradicts \eqref{disintegr}.
\end{proof}

By normalizing the Haar weights, if necessary, we may (and will) assume $\theta(\I_\HH)=1$. We will write $\Ltwo(\GG/\HH)$ for the GNS Hilbert space of $\theta$ (it was denoted $\sH_\theta$ in Section \ref{kusrep}). The weight $\theta$ determines a unitary implementation $\Upsilon$ of the action $\rho_{\GG/\HH}$ of $\GG$ on $\GG/\HH$. The precise form of the unitary $\Upsilon$ can be deduced as follows.

Define an operator $\Upsilon^*:\Ltwo(\GG)\tens\Ltwo(\GG/\HH)\to\Ltwo(\GG)\tens\Ltwo(\GG/\HH)$ by
\begin{equation}\label{defups}
\Upsilon^*\bigl(\Lambda_{\ph_\GG}(x)\tens\Lambda_\theta(y)\bigr)=(\Lambda_{\ph_\GG}\tens\Lambda_\theta)\bigl(\rho_{\GG/\HH}(y)(x\tens\I_{\Linf(\GG/\HH)})\bigr),\qqquad{x}\in\gN_{\ph_\GG},\:y\in\gN_\theta.
\end{equation}
Then $\Upsilon^*$ is easily seen to be an isometry satisfying the following two conditions:
\begin{itemize}
\item $\Upsilon^*\in\Linf(\GG)\vtens\B(\Ltwo(\GG/\HH))$,
\item $(\Delta_\GG\tens\id)(\Upsilon^*)=\Upsilon^*_{23}\Upsilon^*_{13}$.
\end{itemize}
Hence, by \cite[Corollary 4.11]{Brannanetal} $\Upsilon^*$ is unitary and satisfies the equality
\[
\Delta_\GG(x)=\Upsilon^*(\I\tens{x})\Upsilon,\qqquad{x}\in\Linf(\GG/\HH).
\]
Since $\theta$ is invariant (Lemma \ref{theta-inv}), it follows by the same methods as those used in \cite[Proposition 4.3]{Vaes-implementation} that $\Upsilon$ is the canonical implementation of the action $\rho_{\GG/\HH}$ of $\GG$ on $\Linf(\GG/\HH)$.

Before proving the main result of this section, Theorem \ref{coamenability_dropped}, we recall the following fact from \cite{KKS}. Suppose $\HH\subset\GG$ is identified as an open quantum subgroup via the Hopf $*$-homomorphism $\pi^\uu\in\Mor(\C_0^\uu(\GG),\C_0^\uu(\HH))$ as in Definition \ref{Definition:open} and let $\hh{\pi}^\uu\in\Mor(\C_0^\uu(\hh{\HH}),\C_0^\uu(\hh{\GG}))$ be the dual of $\pi^\uu$. Then $\hh{\pi}^\uu\bigl(\C_0^\uu(\hh{\HH})\bigr)\subset\C_0^\uu(\hh{\GG})$ by \cite[Lemma 2.5]{KKS}.

\begin{theorem}\label{coamenability_dropped}
Let $\GG$ be a locally compact quantum group and let $\HH\subset\GG$ be an open quantum subgroup identified via
the Hopf $*$-homomorphism $\pi^\uu\in\Mor(\C_0^\uu(\GG),\C_0^\uu(\HH))$. Then the $*$-homomorphism $\hh{\pi}^\uu:\C_0^\uu(\hh{\HH})\to\C_0^\uu(\hh{\GG})$ is injective.
\end{theorem}

\begin{proof}
We will conclude injectivity of $\hh{\pi}^\uu$ by proving surjectivity of the dual map
\[
\mbox{$\hh{\pi}^\uu$}^*:\C_0^\uu(\hh{\GG})^*\longrightarrow\C_0^\uu(\hh{\HH})^*.
\]
Take any state $\varrho\in\C_0^\uu(\hh{\HH})^*$. We want to show that there exists $\vartheta\in\C_0^\uu(\hh{\GG})^*$ such that
\begin{equation}\label{rhotheta}
\varrho=\vartheta\comp\hh{\pi}^\uu.
\end{equation}
Note that $\varrho$ can be written as $\varrho=\omega_{\Omega_\varrho}\comp\pi_\varrho$, where $(\sH_\varrho,\pi_\varrho,\Omega_{\varrho})$ is the GNS triple for $\varrho$. Thus for $\eta_1,\eta_2\in\Ltwo(\HH)$ we have
\[
\varrho\bigl((\omega_{\eta_1,\eta_2}\tens\id)(\wW^\HH)\bigr)
=\scp{\eta_1\tens\Omega_\varrho}{((\id\tens\pi_\varrho)(\wW^\HH))(\eta_2\tens\Omega_\varrho)}
=\scp{\eta_1\tens\Omega_\varrho}{U(\eta_2\tens\Omega_\varrho)},
\]
where $U=(\id\tens\pi_\varrho)(\wW^\HH)\in\Rep(\HH,\sH_\varrho)$. We will find a representation $\widetilde{U}$ of $\GG$ on a Hilbert space $\widetilde{\sH}$ and a vector $\xi\in\sH$ such that
\begin{equation}\label{etaeta}
\scp{\eta_1\tens\xi}{\widetilde{U}(\eta_2\tens\xi)}=\scp{\eta_1\tens\Omega_\varrho}{U(\eta_2\tens\Omega_\varrho)},\qqquad\eta_1,\eta_2\in\Ltwo(\HH)
\end{equation}
(note that the left leg of $\widetilde{U}$ can act on elements of $\Ltwo(\HH)$ because $\Ltwo(\HH)\subset\Ltwo(\GG)$). Since $\widetilde{U}$ is necessarily of the form $\widetilde{U}=(\id\tens\phi_{\widetilde{U}})(\wW^\GG)$ for a unique representation $\phi_{\widetilde{U}}$ of $\C_0^\uu(\hh{\GG})$ on $\widetilde{\sH}$, we can rewrite the left hand side of \eqref{etaeta} as
\[
(\omega_\xi\comp\phi_{\widetilde{U}})\bigl((\omega_{\eta_1,\eta_2}\tens\id)(\wW^\GG)\bigr)
=(\omega_\xi\comp\phi_{\widetilde{U}})\Bigl(\hh{\pi}^\uu\bigl((\omega_{\eta_1,\eta_2}\tens\id)(\wW^\HH)\bigr)\Bigr)
\]
where the last equality follows from \eqref{WW} and the way $\Ltwo(\HH)$ is identified with a subspace of $\Ltwo(\GG)$. Therefore with $\vartheta=\omega_\xi\comp\phi_{\widetilde{U}}$ we get equality \eqref{rhotheta} on a dense subset of $\C_0^\uu(\hh{\HH})$ (cf.~\cite[Eq.~(5.14)]{mu2})

We let now $\widetilde{\sH}$ be $\IndK(\sH_\varrho)$ and $\widetilde{U}$ be $\IndK(U)$. As noted before statement of the theorem, the canonical implementation $\Upsilon$ of $\rho_{\GG/\HH}$ used in the definition of $\IndK(U)$ can be explicitly described. We will use formula \eqref{defups} for the adjoint of the canonical implementation $\Upsilon$ of $\rho_{\GG/\HH}$ entering the definition of $\IndK(U)$ and also other notation discussed in Subsection \ref{kusrep}.

Considering $\Linf(\HH)$ as a subalgebra of $\Linf(\GG)$ we have $(\alpha_\HH\tens\id)(U^*) =U_{23}^*U_{13}^*$. In particular $U^*\in\cP$. Let $\xi=(U^*)_*\bigl(\Lambda_\theta(\I_\HH)\tens\Omega_\varrho\bigr)\in\IndK(\sH_\varrho)$. Then for all $\eta\in\Ltwo(\HH)\subset\Ltwo(\GG)$ and $x\in\gN_{\ph_\HH}$ we have
{\allowdisplaybreaks
\begin{align*}
&\scp{\eta\tens\xi}{\IndK(U)^*\bigl(\Lambda_{\ph_\HH}(x)\tens\xi\bigr)}\\
&\qqquad=\scp{\eta\tens(U^*)_*\bigl(\Lambda_\theta(\I_\HH)\tens\Omega_\varrho\bigr)}
{\IndK(U)^*\bigl(\Lambda_{\ph_\HH}(x)\tens(U^*)_*\bigl(\Lambda_\theta(\I_\HH)\tens\Omega_\varrho\bigr)\bigr)}\\
&\qqquad\stackrel{\scriptscriptstyle{1}}{=}\scp{\eta\tens(U^*)_*\bigl(\Lambda_\theta(\I_\HH)\tens\Omega_\varrho\bigr)}
{\bigl((\Delta_\GG\tens\id)(U^*)\bigr)_*\Upsilon_{12}^*\bigl(\Lambda_{\ph_\HH}(x)\tens\Lambda_\theta(\I_\HH)\tens\Omega_\varrho\bigr)}\\
&\qqquad=\scp{\eta\tens\Lambda_\theta(\I_\HH)\tens\Omega_\varrho}
{\bigl(\I\tens(U^*)_*\bigr)^*\bigl((\Delta_\GG\tens\id)(U^*)\bigr)_*
\Upsilon_{12}^*\bigl(\Lambda_{\ph_\HH}(x)\tens\Lambda_\theta(\I_\HH)\tens\Omega_\varrho\bigr)}\\
&\qqquad\stackrel{\scriptscriptstyle{2}}{=}\scp{\eta\tens\Lambda_\theta(\I_\HH)\tens\Omega_\varrho}
{U_{23}(\Delta_\GG\tens\id)(U^*)\Upsilon_{12}^*\bigl(\Lambda_{\ph_\HH}(x)\tens\Lambda_\theta(\I_\HH)\tens\Omega_\varrho\bigr)}\\
&\qqquad\stackrel{\scriptscriptstyle{3}}{=}\scp{\eta\tens\Lambda_\theta(\I_\HH)\tens\Omega_\varrho}
{U_{23}(\I\tens\I_\HH\tens\I)(\Delta_\GG\tens\id)(U^*)\Upsilon_{12}^*\bigl(\Lambda_{\ph_\HH}(x)\tens\Lambda_\theta(\I_\HH)\tens\Omega_\varrho\bigr)}\\
&\qqquad=\scp{\eta\tens\Lambda_\theta(\I_\HH)\tens\Omega_\varrho}
{U_{23}(\Delta_\HH\tens\id)(U^*)\Upsilon_{12}^*\bigl(\Lambda_{\ph_\HH}(x)\tens\Lambda_\theta(\I_\HH)\tens\Omega_\varrho\bigr)}\\
&\qqquad=\scp{\eta\tens\Lambda_\theta(\I_\HH)\tens\Omega_\varrho}
{U_{23}U_{23}^*U_{13}^*\Upsilon_{12}^*\bigl(\Lambda_{\ph_\HH}(x)\tens\Lambda_\theta(\I_\HH)\tens\Omega_\varrho\bigr)}\\
&\qqquad=\scp{\eta\tens\Lambda_\theta(\I_\HH)\tens\Omega_\varrho}
{U_{13}^*\Upsilon_{12}^*\bigl((\Lambda_{\ph_\HH}\tens\Lambda_\theta)(x\tens\I_\HH)\tens\Omega_\varrho\bigr)}\\
&\qqquad=\scp{\eta\tens\Lambda_\theta(\I_\HH)\tens\Omega_\varrho}
{U_{13}^*\Upsilon_{12}^*\bigl((\Lambda_{\ph_\GG}\tens\Lambda_\theta)(x\tens\I_\HH)\tens\Omega_\varrho)\bigr)}\\
&\qqquad\stackrel{\scriptscriptstyle{4}}{=}\scp{\eta\tens\Lambda_\theta(\I_\HH)\tens\Omega_\varrho}
{U_{13}^*\bigl((\Lambda_{\ph_\GG}\tens\Lambda_\theta)\bigl(\Delta_\GG(\I_\HH)(x\tens\I)\bigr)\tens\Omega_\varrho\bigr)}\\
&\qqquad\stackrel{\scriptscriptstyle{5}}{=}\scp{\eta\tens\Lambda_\theta(\I_\HH)\tens\Omega_\varrho}
{U_{13}^*(\I_\HH\tens\I\tens\I)\bigl((\Lambda_{\ph_\GG}\tens\Lambda_\theta)\bigl(\Delta_\GG(\I_\HH)(x\tens\I)\bigr)\tens\Omega_\varrho\bigr)}\\
&\qqquad=\scp{\eta\tens\Lambda_\theta(\I_\HH)\tens\Omega_\varrho}
{U_{13}^*\bigl((\Lambda_{\ph_\GG}\tens\Lambda_\theta)\bigl((\I_\HH\tens\I)\Delta_\GG(\I_\HH)(x\tens\I)\bigr)\tens\Omega_\varrho\bigr)}\\
&\qqquad=\scp{\eta\tens\Lambda_\theta(\I_\HH)\tens\Omega_\varrho}
{U_{13}^*\bigl((\Lambda_{\ph_\GG}\tens\Lambda_\theta)\bigl((\I_\HH\tens\I_\HH)(x\tens\I)\bigr)\tens\Omega_\varrho\bigr)}\\
&\qqquad=\theta(\I_\HH)\scp{\eta\tens\Omega_\varrho}{U^*\bigl((\Lambda_{\ph_\GG}(x)\tens\Omega_\varrho\bigr)}
=\scp{\eta\tens\Omega_\varrho}{U^*\bigl((\Lambda_{\ph_\GG}(x)\tens\Omega_\varrho\bigr)},
\end{align*}
}where in $\stackrel{\scriptscriptstyle{1}}{=}$ we used the definition of $\IndK(U)$, equalities $\stackrel{\scriptscriptstyle{3}}{=}$ and $\stackrel{\scriptscriptstyle{5}}{=}$ follow from the fact that $U\in\Linf(\HH)\vtens\B(\sK)$, $\stackrel{\scriptscriptstyle{2}}{=}$ is implied by \eqref{props1}, and in $\stackrel{\scriptscriptstyle{4}}{=}$ we used \eqref{defups}.

Now with $\eta_2=\eta$ and $\eta_1=\Lambda_{\ph_\HH}(x)$ we obtain \eqref{etaeta} for all $\eta_2\in\Ltwo(\HH)$ and $\eta_1$ in the range of $\Lambda_{\ph_\HH}$ which suffices  to have \eqref{etaeta} for all $\eta_1,\eta_2\in\Ltwo(\HH)$.
\end{proof}

In order to complete the setup of Rieffel's induction process, we also need a conditional expectation from $\C_0^\uu(\hh{\GG})$ onto $\C_0^\uu(\hh{\HH})$. In fact it is is already implicitly provided by the proof of Theorem \ref{coamenability_dropped}. Let us recall that for any von Neumann algebra $\sM$ the predual $\sM_*$ is a bimodule over $\sM$ in a natural way. In the following proposition we will use this structure for $\sM=\Linf(\GG)$.

\begin{proposition}\label{cond-exp}
The map
\begin{equation}\label{condE}
\lambda^\uu(\omega)\longmapsto\lambda^\uu(\I_\HH\cdot\omega),\qqquad\omega\in\Lone(\GG)
\end{equation}
extends to a conditional expectation $E:\C_0^\uu(\hh{\GG})\to\hh{\pi}^\uu\bigl(\C_0^\uu(\hh{\HH})\bigr)$.
\end{proposition}

\begin{proof}
The mapping \eqref{condE} is explicitly given by the formula
\begin{equation}\label{condE2}
(\omega\tens\id)(\wW^\GG)\longmapsto(\omega\cdot\I_\HH\tens\id)(\wW^\GG),\qqquad\omega\in\Lone(\GG)
\end{equation}
and since
\[
\begin{split}
(\I_\HH\tens\I)\wW^\GG&=(\pi\tens\id)(\wW^\GG)=(\pi\tens\id)((\Lambda_\GG\tens\id)(\WW^\GG))\\
&=(\Lambda_\HH\tens\id)((\pi^\uu\tens\id)(\WW^\GG))=(\Lambda_\HH\tens\id)((\id\tens\hh{\pi}^\uu)(\WW^\HH))
=(\id\tens\hh{\pi}^\uu)(\wW^\HH),
\end{split}
\]
we have
\[
(\omega\cdot\I_\HH\tens\id)(\wW^\GG)=\hh{\pi}^\uu\bigl((\omega\tens\id)(\wW^\HH)\bigr).
\]
Thus \eqref{condE} is defined on a dense subset of $\C_0^\uu(\hh{\GG})$ and maps onto a dense subset of $\C_0(\hh{\HH})$. To see that it is bounded we notice that by Theorem \ref{coamenability_dropped} there exists $\vartheta\in\C_0^\uu(\hh{\GG})^*$ such that $\eps_{\hh{\HH}}=\vartheta\comp\hh{\pi}^\uu$, where $\eps_{\hh{\HH}}$ is the counit of $\hh{\HH}$. Thus
\[
\begin{split}
(\omega\cdot\I_\HH\tens\id)(\wW^\GG)&=(\omega\tens\id)\bigl(\bigl([(\id\tens\vartheta)(\wW^\GG)]\tens\I\bigr)\wW^\GG\bigr)\\
&=(\omega\tens\id\tens\vartheta)(\wW^\GG_{13}\wW^\GG_{12})\\
&=(\id\tens\vartheta)\Delta_{\hh{\GG}}^\uu\bigl((\omega\tens\id)(\wW^\GG)\bigr)
\end{split}
\]
and consequently the mapping \eqref{condE2} is the restriction of the bounded map
\[
\C_0(\hh{\GG})\ni{x}\longmapsto(\id\tens\vartheta)\Delta_{\hh{\GG}}^\uu(x).
\]
It is a conditional expectation due to the fact that $\I_\HH$ is a central idempotent.
\end{proof}

In view of Theorem \ref{coamenability_dropped} we often identify $\C_0^\uu(\hh{\HH})$ with $\hh{\pi}^\uu\bigl(\C_0^\uu(\hh{\HH})\bigr)\subset \C_0^\uu(\hh{\GG})$.

Let us also comment that \cite[Theorem 5.8]{KKS} provides a converse of Theorem \ref{coamenability_dropped} for regular locally compact quantum groups.

\section{Imprimitivity type result for Rieffel induction} \label{Sec:imprim}

In this section we prove an imprimitivity result for Rieffel's induction in our context, which will hold a key role in the main theorems of the paper to be established in the following section. This should be compared to \cite[Theorem 3.8]{rief1}.

In what follows $\GG$ is a locally compact quantum group, $\HH$ is an open quantum subgroup of $\GG$, and $E:\C_0^\uu(\hh{\GG})\to\C_0^\uu(\hh{\HH})$ is the conditional expectation constructed in Section \ref{sect:RepTh}. For the motivation of the definition below see also the last part of Subsection \ref{vaesrep}.

\begin{definition}
Let $U\in\M(\C_0(\GG)\tens\cK(\cE))$ be a unitary representation of $\GG$ on a Hilbert $B$-module $\cE$. A \emph{$U$-covariant representation of $\Linf(\GG/\HH)$} on $\cE$ is a strict unital $*$-homomorphism $\rho:\Linf(\GG/\HH)\to\cL(\cE)$ that satisfies
\[
U^*\bigl(\I\tens\rho(x)\bigr)U=(\id\tens\rho)(\Delta_\GG(x))
\]
for all $x\in\Linf(\GG/\HH)$.
\end{definition}

We begin by a preliminary lemma.

\begin{lemma}\label{1_H-cond.exp.-1_H}
Suppose $U\in\M(\C_0(\GG)\tens\cK(\cE))$ is a unitary representation of $\GG$ on a Hilbert $B$-module $\cE$ and let $\rho$ be a $U$-covariant representation of $\Linf(\GG/\HH)$ on $\cE$. Then
\[
\rho(\I_\HH)\phi_U(x)\rho(\I_\HH)=\rho(\I_\HH)\phi_U\bigl(E(x)\bigr)\rho(\I_\HH),\qqquad{x}\in\C_0^\uu(\hh{\GG}),
\]
where $\phi_U:\C_0^\uu(\hh{\GG})\to\cL(\cE)$ is the representation corresponding to $U$.
\end{lemma}

\begin{proof}
Using covariance of $\rho$ and the group-like property of $\I_\HH$, for every $\omega\in \Lone(\GG)$ we have
{\allowdisplaybreaks
\begin{align*}
\rho(\I_\HH)\phi_U\bigl(\lambda^\uu(\omega)\bigr)\rho(\I_\HH)
&=(\omega\tens\id)\Bigl(\bigl(\I\tens\rho(\I_\HH)\bigr)U\bigl(\I\tens\rho(\I_\HH)\bigr)\Bigr)\\
&=(\omega\tens\id)\Bigl(U\bigl((\id\tens\rho)\Delta_\GG(\I_\HH)\bigr)\bigl(\I\tens\rho(\I_\HH)\bigr)\Bigr)\\
&=(\omega\tens\id)\Bigl(U\bigl((\id\tens\rho)\bigl(\Delta_\GG(\I_\HH)(\I\tens\I_\HH)\bigr)\bigr)\Bigr)\\
&=(\omega\tens\id)\Bigl(U\bigl(\I_\HH\tens\rho(\I_\HH)\bigr)\Bigr)\\
&=\bigl((\I_\HH\cdot{\omega}\tens\id)(U)\bigr)\rho(\I_\HH).
\end{align*}
}But the same calculations also give
\[
\rho(\I_\HH)\phi_U\bigl(\lambda^\uu(\I_\HH\cdot\omega)\bigr)\rho(\I_\HH)=
\bigl((\I_\HH\cdot\omega)\tens\id)U\bigr)\rho(\I_\HH),
\]
and hence lemma follows from Proposition \ref{cond-exp}.
\end{proof}

The following is the main result of this section.

\begin{theorem}\label{condition}
Let $U\in\M(\C_0(\GG)\tens\cK(\cE))$ be a unitary representation of $\GG$ on a  Hilbert $B$-module $\cE$ and suppose that $\rho$ is a $U$-covariant representation of $\Linf(\GG/\HH)$ on $\cE$. Then the submodule $\cE_0=\rho(\I_\HH)\cE$ is $\C_0^\uu(\hh{\HH})$-invariant, so that there is a natural representation of $\HH$ on $\cE_0$. Furthermore there is a unitary $T\in\cL(\IndR(\cE_0),\cE)$ intertwining the respective actions of $\C_0^\uu(\hh{\GG})$.
\end{theorem}

\begin{proof}
Let $\phi_U:\C_0^\uu(\hh{\GG})\to\cL(\cE)$ be the representation corresponding to $U$. By covariance of $\rho$ we have
\[
(\I_\HH\tens\I)U^*\bigl(\I\tens\rho(\I_\HH)\bigr)U(\I_\HH\tens\I)=(\id\tens\rho)\bigl((\I_\HH\tens\I)\Delta_\GG(\I_\HH)(\I_\HH\tens\I)\bigr)
=\I_\HH\tens\rho(\I_\HH),
\]
which implies $(\I\tens\rho(\I_\HH))U(\I_\HH\tens \I) = U(\I_\HH\tens \I)(\I\tens\rho(\I_\HH))$, and hence
\[
\rho(\I_\HH)(\omega\tens\id)\bigl(U(\I_\HH\tens\I)\bigr)=(\omega\tens\id)\bigl(U(\I_\HH\tens\I)\bigr)\rho(\I_\HH)
\]
for all $\omega\in\Lone(\GG)$. The subspace ${\bigl\{(\omega\tens\id)\bigl(U(\I_\HH\tens\I)\bigr)\st}\omega\in\Lone(\GG)\bigr\}$ is norm dense in $\phi_U\bigl(\C_0^\uu(\hh{\HH})\bigr)$ because $\lambda^\uu(\Lone(\GG))$ is dense in $\C_0^\uu(\hh{\GG})$, and therefore the first part of the statement follows.

For $a\in\C_0^\uu(\hh{\GG})$ and $v\in\cE$ we will denote by $a\dtens{v}$ the class of $a\tens{v}$ in $\IndR(\cE)$. We claim the map $T$ defined by
\begin{equation}\label{defT}
T(a\dtens{v})=\phi_U(a)v,\qqquad{a}\in\C_0^\uu(\hh{\GG}),\:v\in\cE,
\end{equation}
extends to a unitary between $\IndR(\cE_0)$ and $\cE$.

First, to see that $T$ is an isometry, apply Lemma \ref{1_H-cond.exp.-1_H} to observe that for $a,b\in\C_0^\uu(\hh{\GG})$ and $v,v'\in\cE_0$
\[
\begin{split}
\scp{\phi_U(a)v}{\phi_U(b)v'}_{\cE}&=\scp{v}{\rho(\I_\HH)\phi_U(a^*b)\rho(\I_\HH)v'}_{\cE}\\
&=\scp{v}{\phi_U\bigl(E(a^*b)\bigr)v'}_{\cE}\\
&=\scp{a\dtens{v}}{b\dtens{v'}}_{\IndR(\cE_0)},
\end{split}
\]

Next, in order to prove surjectivity of $T$ we show that the set
\begin{equation}\label{lindens1}
{\bigl\{\phi_U(a)\rho(\I_\HH)\phi_U(b)v\st}{a,b}\in\C_0^\uu(\hh{\GG}),\:v\in\cE\bigr\},
\end{equation}
clearly contained in the image of $T$, spans a dense subspace of $\cE$. By \cite[Theorem 3.3]{KKS} and covariance of $\rho$ we get
\begin{align*}
\I_{\cE}=\rho(\I_{\GG/\HH})
&\in\Lin\bigl\{\rho\bigl((\omega_{\xi,\eta}\tens\id)\Delta_\GG(\I_\HH)\bigr)\st\xi,\eta\in\Ltwo(\GG)\bigr\}\\
&=\Lin\bigl\{(\omega_{\xi,\eta}\tens\id)\bigl(U^*(\I\tens\rho(\I_\HH))U\bigr)\st\xi,\eta\in\Ltwo(\GG)\bigr\}.
\end{align*}
Now let $\{e_i\}_{i\in{I}}$ be an orthonormal basis for $\Ltwo(\GG)$. Then for every $\xi,\eta\in\Ltwo(\GG)$ we have
\[
(\omega_{\xi,\eta}\tens\id)\bigl(U^*(\I\tens\rho(\I_\HH))U\bigr)=\sum_{i\in I}\bigl((\omega_{\xi,e_i}\tens\id)(U^*)\bigr)\rho(\I_\HH)
\bigl((\omega_{e_i,\eta}\tens\id)(U)\bigr)
\]
with the sum strictly convergent. Since $(\omega_{\xi,e_i}\tens\id)(U^*)=\phi_U\bigl((\omega_{e_i,\xi}\tens\id)({\wW^\GG}^*)\bigr)$ and $(\omega_{e_i,\eta}\tens\id)(U)=\phi_U\bigl((\omega_{e_i,\eta}\tens\id)(\wW^\GG)\bigr)$ it follows that $\I_\cE$ is in the strict closure of $\bigl\{\phi_U(a)\rho(\I_\HH)\phi_U(b)\st{a,b}\in\C_0^\uu(\hh{\GG})\bigr\}$ and the proof of surjectivity is complete. Theorem 3.5 of \cite{Lance} shows that in fact $T$ is adjointable and unitary (as it is obviously $B$-linear).

Finally, the fact that $T$ intertwines the respective $\C_0^\uu(\hh{\GG})$-actions is obvious from \eqref{defT}.
\end{proof}

\section{Equivalence of induction processes}\label{sect:equiv}

In this section we prove the main result of the paper: for open quantum subgroups all three induction procedures
$\IndK$, $\IndR$, and $\IndV$ are canonically equivalent. As mentioned in the introduction, the key technical tool will be the imprimitivity result of Section \ref{Sec:imprim}.

\subsection{Equivalence of $\IndK$ and $\IndR$}

We first show the equivalence of the Rieffel induction with that of Kustermans.

\begin{theorem}
Let $\HH$ be an open quantum subgroup of $\GG$, and let $U\in\Linf(\HH)\vtens\B(\sK)$ be a unitary representation of $\HH$ on a Hilbert space $\sK$. Then $\IndK(U)$ and $\IndR(U)$ are unitarily equivalent.
\end{theorem}

\begin{proof}
The $*$-homomorphism $\rho:\Linf(\GG/\HH)\to\B(\IndK(\sK))$ defined by
\[
\rho(y)(X\tens\xi\tens{v})=\bigl((y\tens\I)X\bigr)\tens\xi\tens{v},\qqquad{y}\in\Linf(\GG/\HH),\:X\in\cP,\:\xi\in\Ltwo(\GG/\HH),\:v\in\sK
\]
is normal by \cite[Result 4.8]{KustermansInduced}, and moreover for $y,X,\xi$ and $v$ as above we have
{\allowdisplaybreaks
\begin{align*}
\IndK(U)^*\bigl(\I\tens\rho(y)\bigr)&\IndK(U)\Bigl(\bigl(\bigl(\Delta_\GG\tens\id)(U^*)\bigr)_*\Upsilon_{12}^*(\xi\tens{x}\tens{v})\Bigr)\\
&=\IndK(U)^*\bigl(\I\tens\rho(y)\bigr)\bigl(\xi\tens(U^*)_*(x\tens{v})\bigr)\\
&=\IndK(U)^*\Bigl(\xi\tens\bigl((y\tens\I)U^*\bigr)_*(x\tens v)\Bigr)\\
&=\Bigl((\Delta_\GG\tens\id)\bigl((y\tens\I)U^*\bigr)\Bigr)_*\Upsilon_{12}^*(\xi\tens{x}\tens{v})\\
&=\Bigl(\bigl(\Delta_\GG(y)\tens\I\bigr)(\Delta_\GG\tens\id)(U^*)\Bigr)_*\Upsilon_{12}^*(\xi\tens{x}\tens{v})\\
&\stackrel{\scriptscriptstyle{6}}{=}(\id\tens\rho)\bigl(\Delta_\GG(y)\bigr)
\Bigl(\bigl((\Delta_\GG\tens\id)(U^*)\bigr)_*\Upsilon_{12}^*(\xi\tens{x}\tens{v})\Bigr)
\end{align*}
}which shows $\rho$ is $\IndK(U)$-covariant (in $\stackrel{\scriptscriptstyle{6}}{=}$ we used \cite[Result 4.9.(2)]{KustermansInduced}). Thus, in view of Theorem \ref{condition}
it now suffices to show there is a unitary $T:\rho(\I_\HH)\IndK(\sK)\to\sK$ that intertwines $\C_0^\uu(\hh{\HH})$-actions.

For every $X\in\cP$ we have
\[
\begin{split}
(\Delta_\HH\tens\id)\bigl(U(\I_\HH\tens\I)X\bigr)&=U_{13}U_{23}(\I_\HH\tens\I_\HH\tens\I)U_{23}^*X_{13}\\
&=U_{13}(\I_\HH\tens\I_\HH\tens\I)X_{13}
\end{split}
\]
which implies there exists $z^X\in\B(\sK)$ such that
\begin{equation}\label{eq1_XU}
U(\I_\HH\tens\I)X=\I_\HH\tens{z^X}.
\end{equation}
Now define the map $T:\rho(\I_\HH)\IndK(\sK)\to\Ltwo(\GG/\HH)\tens\sK$ by
\[
(\I_\HH\tens\I)X\tens\xi\tens{v}\longmapsto\I_\HH\xi\tens{z^X}v,\qqquad{X}\in\cP,\:\xi\in\Ltwo(\GG/\HH),\:v\in\sK.
\]
By \eqref{eq1_XU} $T$ is isometric. Since $\I_{\HH}$ is a minimal projection in $\B(\Ltwo(\GG/\HH))$, by identifying $\I_\HH\tens\sK\cong\sK$ we may consider $T$ as an isometric map taking values in $\sK$.

Conversely to \eqref{eq1_XU}, observe that for each $z\in\B(\sK)$ we have $U^*(\I_\HH\tens{z})\in\cP$, which implies that $T$ maps onto $\sK$, and hence is a unitary from $\rho(\I_\HH)\IndK(\sK)$ onto $\sK$. Its inverse is given by the formula
\begin{equation}\label{T*}
T^*(\I_\HH\xi\tens{v})=U^*\tens\I_\HH\xi\tens{v}=(U^*)_*(\I_\HH\xi\tens{v}).
\end{equation}
It remains to show $T$ intertwines the $\C_0^\uu(\hh{\HH})$-actions. First, recall the unitary $\Upsilon^*$ defined in \eqref{defups}, and note that since $\I_\HH\in\Linf(\GG/\HH)$ is a minimal central projection, we have
\begin{equation}\label{propups}
(\I_\HH\tens\I_\HH)\Upsilon^*(\eta\tens{y})= \rho_{\GG/\HH}(\I_\HH y)(\I_\HH\eta\tens\I_\HH)=(\I_\HH\tens\I_\HH)(\eta\tens{y})
\end{equation}
for all $\eta\in\Ltwo(\GG)$ and $y\in\gN_\theta$.

For $\eta\in\Ltwo(\GG)$, $\xi\in\Ltwo(\GG/\HH)$, $v\in\sK$ and $z\in\B(\sK)$, let $\zeta=\I_\HH\eta\tens\bigl(U^*(\I_\HH\tens{z})\bigr)\tens\xi\tens{v}$. Then
{\allowdisplaybreaks
\begin{align*}
\IndK(U)^*(\zeta)
&=\bigl((\Delta_\GG\tens\id)\bigl(U^*(\I_\HH\tens{z})\bigr)\bigr)_*\bigl(\Upsilon^*(\I_\HH\eta\tens\xi)\tens{v}\bigr)\\
&=\bigl((\Delta_\GG\tens\id)\bigl(U^*(\I_\HH\tens z)\bigr)\bigr)_*\bigl((\I_\HH\tens\I\tens\I)\bigl(\Upsilon^*(\eta\tens\xi)\tens{v}\bigr)\bigr)\\
&\stackrel{\scriptscriptstyle{7}}{=}
\Bigl(\bigl((\Delta_\GG\tens\id)\bigl(U^*(\I_\HH\tens z)\bigr)\bigr)(\I_\HH\tens\I\tens\I)\Bigr)_*\bigl(\Upsilon^*(\eta\tens\xi)\tens{v}\bigr)\\
&=\Bigl((\Delta_\HH\tens\id)\bigl(U^*(\I_\HH\tens{z})\bigr)\Bigr)_*\bigl(\Upsilon^*(\eta\tens\xi)\tens{v}\bigr)\\
&=\bigl(U^*_{23}U^*_{13}(\I_\HH\tens\I_\HH\tens{z})\bigr)_*\bigl(\Upsilon^*(\eta\tens\xi)\tens{v}\bigr)\\
&\stackrel{\scriptscriptstyle{7}}{=}
(U^*_{23}U^*_{13})_*(\I_\HH\tens\I_\HH\tens{z})\bigl(\Upsilon^*(\eta\tens\xi)\tens{v}\bigr)\\
&\stackrel{\scriptscriptstyle{8}}{=}
(U^*_{23}U^*_{13})_*(\I_\HH\tens\I_\HH\tens{z})(\eta\tens\xi\tens{v})\\
&\stackrel{\scriptscriptstyle{7}}{=}
(U^*_{23})_*U^*_{13}(\I_\HH\tens\I_\HH\tens{z})(\eta\tens\xi\tens{v})\\
&=(U^*_{23})_*U^*_{13}(\I_\HH\eta\tens\I_\HH\xi\tens{zv})\\
&=(U^*_{23})_*U^*(\id\tens{T})(\zeta)\\
&\stackrel{\scriptscriptstyle{9}}{=}
(\id\tens{T^*})U^*(\id\tens{T})(\zeta),
\end{align*}
}where in the equalities $\stackrel{\scriptscriptstyle{7}}{=}$ we applied \eqref{props2}, for $\stackrel{\scriptscriptstyle{8}}{=}$ we used \eqref{propups}, and in $\stackrel{\scriptscriptstyle{9}}{=}$ uses \eqref{T*}. This completes the proof.
\end{proof}

\subsection{Equivalence of $\IndR$ and $\IndV$}

Here we will proceed as above to show that the Rieffel and Vaes induction procedures are equivalent in the case of open quantum subgroups, also for representations on Hilbert modules.

Let for a moment $\HH$ be a closed quantum subgroup of $\GG$, $B$ be a \cst-algebra, and $\cE$ a Hilbert $B$-module. Recall the definition of $\cI$ and the map $\pil:\Linf(\hh{\HH})\to\cL(\Ltwo(\GG)\tens\cE)$ satisfying \eqref{Defpli} from Section \ref{vaesrep}, and observe that
\begin{equation}\label{pi_ell-relation}
\begin{split}
(\pil\tens\id)\comp(\id\tens\hh{\pi})\comp\Delta_{\hh{\HH}}
&=(\hh{\pi}\tens\phi_U\tens\hh{\pi})\comp(\Delta_{\hh{\HH}^\op}\tens\id)\comp\Delta_{\hh{\HH}}\\
&=\flip_{23}\comp\bigl((\hh{\pi}\tens\hh{\pi}\tens\phi_U)\comp(\Delta_{\hh{\HH}}\tens\id)\comp\Delta_{\hh{\HH}^\op}\bigr)\\
&=\flip_{23}\comp\Bigl((\Delta_{\hh{\GG}}\tens\id)\comp\bigl((\hh{\pi}\tens\phi_U)\Delta_{\hh{\HH}^\op}\bigr)\Bigr)\\
&=\flip_{23}\comp\bigl((\Delta_{\hh{\GG}}\tens\id)\comp\pil\bigr).
\end{split}
\end{equation}

We are ready for the main result of this subsection.

\begin{theorem}\label{mainres}
Let $\HH$ be an open quantum subgroup of $\GG$, and let $U\in\M(\C_0(\HH)\tens\cK(\cE_0))$ be a unitary representation of $\HH$ on a Hilbert $B$-module $\cE_0$. Then $\IndK(U)$ and $\IndR(U)$ are unitarily equivalent.
\end{theorem}

\begin{proof}
Recall the map \eqref{covrepVaes} from Subsection \ref{vaesrep}. In light of Proposition \ref{condition} we only need to show that there is a unitary equivalence $T:\rho(\I_\HH)\IndV(\cE_0)\to\cE_0$ which intertwines the restrictions of $\C_0^\uu(\hh{\GG})$ actions to $\C_0^\uu(\hh{\HH})$; note that strictness of $\rho:\Linf(\GG/\HH)\to\cL(\IndV(\cE_0))$ was proved in \cite{Vaes-induction}.

Applying \eqref{covrepVaes} to the identification \eqref{box-iden} we get
\[
\Ltwo(\GG)\tens\bigl(\rho(\I_\HH)\IndV(\cE_0)\bigr)\cong\I_\HH\cI\btens_{\pil}\Ltwo(\GG)\tens\cE_0 .
\]
But for every $v\in\cI$ if we consider $x=\I_\HH{v}$ as a map on $\Ltwo(\HH)=\I_\HH\Ltwo(\GG)$, since $\Linf(\GG/\HH)\subset\hh{\pi}'\bigl(\Linf(\hh{\HH})'\bigr)$ (cf.~remarks after \cite[Definition 4.1]{Vaes-induction}), it follows that $xy=yx$ for all $y\in\Linf(\hh{\HH})'$ and hence $\I_\HH{v}\in\Linf(\hh{\HH})$. On the other hand it is obvious that every $x\in\Linf(\hh{\HH})$ is of the form $x=\I_\HH{v}$ for some $v\in\cI$.

Hence, we have the identifications
\begin{equation}\label{reps-ident}
\begin{split}
\Ltwo(\GG)\tens\bigl(\rho(\I_\HH)\IndV(\cE_0)\bigr)&\cong\I_\HH\cI\btens_{\pil}\bigl(\Ltwo(\GG)\tens\cE_0\bigr)\\
&\cong\Linf(\hh{\HH})\btens_{\pil}\bigl(\Ltwo(\GG)\tens\cE_0\bigr)\\
&\cong\Ltwo(\GG)\tens\cE_0 .
\end{split}
\end{equation}
Now our goal is to show the above identification $\Ltwo(\GG)\tens\bigl(\rho(\I_\HH)\IndV(\cE_0)\bigr)\cong\Ltwo(\GG)\tens\cE_0$ is given by a map of the form $\id\tens{T}$, where $T:\rho(\I_\HH)\IndV(\cE_0)\to\cE_0$ is as desired above. For this, we will show the latter identification commutes with the right action of $\Linf(\hh{\GG})\tens\I$ and the canonical action of $\Linf(\GG)'\tens\I$, and hence with $\B(\Ltwo(\GG))\tens\I$.

The first two identifications in \eqref{reps-ident} preserve the corresponding $\bicorresp{{\;\bullet\;}}{\Linf(\hh{\GG})}{\Linf(\hh{\GG})}{\Linf(\GG)'}$ structures. Therefore, it remains to show that the last identification
\[
\iota:\Linf(\hh{\HH})\btens_{\pil}\bigl(\Ltwo(\GG)\tens\cE_0\big)\ni\hh{a}\tens\zeta\longmapsto\pil(\hh{a})\zeta\in\Ltwo(\GG)\tens\cE_0
\]
commutes with $\B(\Ltwo(\GG))\tens\I$.

It is clear that the right actions of $\Linf(\hh{\GG})$ on $\Linf(\hh{\HH})\btens\limits_{\pil}\Ltwo(\GG)\tens\cE_0$ and $\Ltwo(\GG)\tens\cE_0$ are intertwined by $\iota$. To see that $\iota$ also intertwines the action of $\Linf(\GG)'$, let $x'=(\id\tens\omega)\vv^{\hh{\GG}}\in\Linf(\hh{\GG})'$, $v\in\cI$, and
$\zeta\in\Ltwo(\GG)\tens\cE_0$ and note that we have
{\allowdisplaybreaks
\begin{align*}
\iota\bigl(\pi'(x')&(\I_\HH{v}\btens_{\pil}\zeta)\bigr)=\iota\bigl(\pi'\bigl((\id\tens\omega)(\vv^{\hh{\GG}})\bigr)(\I_\HH{v}\btens_{\pil}\zeta)\bigr)\\
&\stackrel{\scriptscriptstyle{10}}{=}
\iota\bigl(\bigl((\id\tens\omega)(\hh{U})\bigr)(\I_\HH v\btens_{\pil}\zeta)\bigr)
\stackrel{\scriptscriptstyle{11}}{=}
(\iota\tens\omega)\bigl(\alpha_{\widetilde{\cF}}(\I_\HH{v}\btens_{\pil}\zeta)\bigr)\\
&=(\iota\tens\omega)\bigl(\alpha_{\cI}(\I_\HH{v})\btens_{\pil\tens\id}\alpha_{\Ltwo(\GG)\tens\cE_0}(\zeta)\bigr)\\
&=(\iota\tens\omega)\bigl(\vv^{\hh{\GG}}(\I_\HH{v}\tens\I)(\id\tens\hh{\pi})(\vv^{\hh{\HH}})^*\btens_{\pil\tens\id}\vv^{\hh{\GG}}_{13}(\zeta\tens\I)\bigr)\\
&=(\iota\tens\omega)
\bigl(\vv^{\hh{\GG}}(\I_\HH{v}\tens\I)(\id\tens\hh{\pi})\bigl((\I_\HH\tens\I_\HH)(\vv^{\hh{\GG}})^*(\I_\HH\tens\I_\HH)\bigr)\btens_{\pil\tens\id}
\vv^{\hh{\GG}}_{13}(\zeta\tens\I)\bigr)\\
&=(\iota\tens\omega)
\bigl((\I_\HH\tens\I)\vv^{\hh{\GG}}(\I_\HH{v}\I_\HH\tens\I)(\vv^{\hh{\GG}})^*(\I_\HH\tens\I)\btens_{\pil\tens\id}\vv^{\hh{\GG}}_{13}(\zeta\tens\I)\bigr)\\
&=(\iota\tens\omega)\bigl((\I_\HH\tens\I)\Delta_{\hh{\GG}}(\I_\HH{v}\I_\HH)(\I_\HH\tens\I)\btens_{\pil\tens\id}\vv^{\hh{\GG}}_{13}(\zeta\tens\I)\bigr)\\
&=(\iota\tens\omega)\bigl((\id\tens\hh{\pi})\Delta_{\hh{\HH}}(\I_\HH{v}\I_\HH)\btens_{\pil\tens\id}\vv^{\hh{\GG}}_{13}(\zeta\tens\I)\bigr)\\
&=(\id\tens\omega)(\iota\tens\id)\bigl((\id\tens\hh{\pi})\Delta_{\hh{\HH}}(\I_\HH{v})\btens_{\pil\tens\id}\vv^{\hh{\GG}}_{13}(\zeta\tens\I)\bigr)\\
&=(\id\tens\omega)
\Bigl(\bigl((\pil\tens\id)(\id\tens\hh{\pi})\Delta_{\hh{\HH}}(\I_\HH{v})\bigr)\vv^{\hh{\GG}}_{13}(\zeta\tens\I)\Bigr)\\
&\stackrel{\scriptscriptstyle{12}}{=}
(\id\tens\omega)
\Bigl(\bigl(\flip_{23}\bigl((\Delta_{\hh{\GG}}\tens\id)\pil(\I_\HH{v})\bigr)\bigr)\vv^{\hh{\GG}}_{13}(\zeta\tens\I)\Bigr)\\
&=(\id\tens\omega)\Bigl(\bigl(\vv^{\hh{\GG}}_{13}\bigl[\pil(\I_\HH{v})\tens\id\bigr](\vv^{\hh{\GG}}_{13})^*\bigr)\vv^{\hh{\GG}}_{13}(\zeta\tens\I)\Bigr)\\
&=(\id\tens\omega)\Bigl(\vv^{\hh{\GG}}_{13}\bigl(\pil(\I_\HH{v})\zeta\tens\I\bigr)\Bigr)\\
&=\Bigl(\bigl[(\id\tens\omega)\vv^{\hh{\GG}}\bigr]\tens\I\Bigr)\bigl(\iota(\I_\HH{v}\btens_{\pil}\zeta)\bigr)\\
&=\pi'(x')\bigl(\iota(\I_\HH{v}\btens_{\pil}\zeta)\bigr),
\end{align*}
}where we used \eqref{hatU} in $\stackrel{\scriptscriptstyle{10}}{=}$, \eqref{hatU-alphaF} in $\stackrel{\scriptscriptstyle{11}}{=}$, and identities \eqref{pi_ell-relation} in $\stackrel{\scriptscriptstyle{12}}{=}$.

Finally, we show $T$ intertwines the actions of $\C_0^\uu(\hh{\HH})$. Note that the identifications in \eqref{reps-ident} are equivariant with respect to the left action of $\Linf(\hh{\HH})$, which in the first three terms is given by the restriction to $\hh{\pi}(\Linf(\hh{\HH}))\subset\Linf(\hh{\GG})$ of the left action of $\Linf(\hh{\GG})$ from their corresponding $\bicorresp{{\;\bullet\;}}{\Linf(\hh{\GG})}{\Linf(\hh{\GG})}{\Linf(\GG)'}$ structures, and in the last term is given by $\pil$. The equivariance of the first identification follows from the construction, and the other are easy to observe.

Now, for $\hh{x}\in\Linf(\hh{\GG})$ denote by $\tilpil(\hh{x})\in\cL(\Ltwo(\GG)\tens\IndV(\cE_0))$ the left action map in the first term of \eqref{reps-ident}. Then, we have
\[
(\id\tens{T})\tilpil\bigl(\hh{\pi}(\hh{a})\bigr)=\pil(\hh{a})(\id\tens{T})
\]
for all $\hh{a}\in\Linf(\hh{\HH})$. Since the second legs of $\tilpil\bigl(\hh{\pi}\bigl(\Linf(\hh{\HH})\bigr)\bigr)$ and $\pil\bigl(\Linf(\hh{\HH})\bigr)$ contain a dense subspace of $\phi_{\IndV(U)}\bigl(\C_0^\uu(\hh{\HH})\bigr)$ and $\phi_U\bigl(\C_0^\uu(\hh{\HH})\bigr)$, respectively, it follows that $T$ intertwines the actions of $\C_0^\uu(\hh{\HH})$, which completes the proof.
\end{proof}

As an immediate consequence of the above theorem, we observe that one may drop strong regularity assumption in the Imprimitivity Theorem for $\IndV$ (\cite[Theorem 5.1]{Vaes-induction}), when the subgroup in question is open.

\begin{theorem}
Let $\GG$ be a locally compact quantum group and $\HH$ an open quantum subgroup of $\GG$. A unitary representation $U\in\Rep(\GG,\cF)$ of $\GG$ is induced
from a representation of $\HH$ in the sense of Vaes if and only if there is a strictly continuous covariant representation of $\Linf(\GG/\HH)$ on the Hilbert module $\cF$.
\end{theorem}

\begin{proof}
The existence of a strictly continuous covariant representation of $\Linf(\GG/\HH)$ for an induced representation was noted in the construction of $\IndV$
in \cite[Section 4]{Vaes-induction}.

Conversely, if $U\in\Rep(\GG,\cF)$ is such that $\cF$ admits a strictly continuous covariant representation of $\Linf(\GG/\HH)$, then by Proposition \ref{condition} there is a unitary representation $U\in\Rep(\HH,\cE)$ of $\HH$ on a Hilbert module $\cE$ such that $(U,\cF)=\bigl(\IndR(V),\IndR(\cE)\bigr)$, and hence the result follows from Theorem \ref{mainres}.
\end{proof}

\renewcommand{\theequation}{\Alph{section}.\arabic{equation}}
\renewcommand{\thesection}{\Alph{section}}
\setcounter{section}{0}
\renewcommand{\theproposition}{\Alph{section}.\arabic{proposition}}
\setcounter{proposition}{0}

\section*{Appendix: Continuity of induction}
\setcounter{section}{1}

The results of this paper in particular imply that all properties of induced representations from open quantum subgroups are shared between all three processes. In particular, in that case all continuity and stability properties proved in \cite{rief1} for $\IndR$ also hold for $\IndK$ and $\IndV$.

A particular important property of $\IndR$ is the continuity, i.e.~preservation of weak containment, which has applications in various representation theoretical problems. As a consequence of Theorem \ref{mainres} the weak containment is preserved under $\IndV$ in the case of induced representations from open quantum subgroups.

In this appendix we show this result holds for $\IndV$ in general, i.e.~for Vaes-induced representations from closed quantum subgroups.

\begin{theorem}\label{ind_cont}
Let $\HH$ be a closed quantum subgroup of a locally compact quantum group $\GG$, and let $U_1\in\Rep(\HH,\sH_1)$ and $U_2\in\Rep(\HH,\sH_2)$ be unitary representations of $\HH$ on Hilbert spaces $\sH_1$ and $\sH_2$. If $U_1\preccurlyeq{U_2}$ then $\IndV(U_1)\preccurlyeq\IndV(U_2)$.
\end{theorem}

\begin{proof}
Consider the universal representation $\wW^\HH\in\Rep(\HH,\C_0^\uu(\hh{\HH}))$ of $\HH$, where $\C_0^\uu(\hh{\HH})$ is seen as
a Hilbert module over itself. Then $\IndV\bigl(\C_0^\uu(\hh{\HH})\bigr)$ is also a Hilbert $\C_0^\uu(\hh{\HH})$-module. Under the canonical identifications
$\displaystyle\C_0^\uu(\hh{\HH})\btens_{\phi_{U_i}}\sH_i \cong \sH_i$, $i =1, 2$, we have $U_i=\wW^{\HH}\btens\limits_{\phi_{U_i}}\I$, and it follows from \cite[Proposition 4.7]{Vaes-induction} that
\[
\IndV(\sH_i)\cong\IndV\bigl(\C_0^\uu(\hh{\HH})\btens_{\phi_{U_i}}\sH_i\bigr)\cong\IndV\bigl(\C_0^\uu(\hh{\HH})\bigr)\btens_{\phi_{U_i}}\sH_i
\]
as $\C_0^\uu(\hh{\GG})$-modules, where in the last term the module action is inherited from the $\C_0^\uu(\hh{\GG})$-module structure of $\IndV\bigl(\C_0^\uu(\hh{\HH})\bigr)$.

Thus, we may now repeat in this context a similar argument to that in the proof of \cite[Proposition 6.26]{rief1} to conclude the result.

Let $\xi\in\IndV\bigl(\C_0^\uu(\hh{\HH})\bigr)$ and $v\in\sH_1$ be such that $\zeta=\xi\tens{v}\in\IndV(\sH_1)$ is of unit length. Denote by $\omega_\zeta$ the corresponding vector state. Then
\[
\omega_\zeta(\hh{a})=\scp{v}{\scp{\xi}{\hh{a}\xi}_{\C_0^\uu(\hh{\HH})}v}
\]
for all $\hh{a}\in\C_0^\uu(\hh{\GG})$. By the assumption there is a net $\bigl\{\nu_j=\sum_{i=1}^{n_j} r_{ij}\omega_{v_{ij}}\bigr\}_{j\in\Jnd}$ of convex combinations of vector states on $\B(\sH_2)$ which converges in the weak${}^*$ topology to $\omega_{v}$. Since
\[
\lim_{j\in\Jnd}\nu_j\bigl(\scp{\xi}{\xi}_{\C_0^\uu(\hh{\HH})}\bigr)=\omega_{v}\bigl(\scp{\xi}{\xi}_{\C_0^\uu(\hh{\HH})}\bigr)
=\scp{\xi\tens{v}}{\xi\tens{v}}=1,
\]
we may assume $\nu_j\bigl(\scp{\xi}{\xi}_{\C_0^\uu(\hh{\HH})}\bigr)\neq{0}$ for all $j\in\Jnd$. For those $\xi\tens{v_{ij}}$ that are nonzero in
$\IndV\bigl(\C_0^\uu(\hh{\HH})\bigr)\btens\limits_{\phi_{U_2}}\sH_2$, we let
\[
\zeta_{ij}=\frac{\xi\tens{v_{ij}}}{\|\xi\tens{v_{ij}}\|}
\]
and set
\[
s_{ij}=\frac{r_{ij}\|\xi\tens{v_{ij}}\|^2}{\nu_j(\scp{\xi}{\xi}_{\C_0^\uu(\hh{\HH})})}.
\]
Then the net $\bigl\{\widetilde{\nu}_j=\sum_is_{ij}\omega_{\zeta_{ij}}\bigr\}_{j\in\Jnd}$ (the sums are taken over non-zero terms) converges in the weak${}^*$ topology to $\omega_\zeta$, and this completes the proof.
\end{proof}

\bibliography{induction}{}
\bibliographystyle{plain}

\end{document}